\documentclass[11pt,twoside]{article}
\usepackage{latexsym}
\usepackage{amssymb,amsbsy,amsmath,amsfonts,amssymb,amscd}
\usepackage{mathrsfs}
\usepackage{epsfig, graphicx}
\usepackage{graphics,graphicx,epsfig,subfigure}
\usepackage{color,verbatim}
\usepackage{authblk}
\usepackage[font=small,format=plain,labelfont=bf,up,textfont=up]{caption} 
\usepackage{subfigure}
\newcommand{\commentout}[1]{}
\newcommand{\R}{\mathbb{R}}
\newcommand{\N}{\mathbb{N}}

\newcommand {\eps}  {\varepsilon}

\newcommand {\Chi} {{\bf \raise 2pt \hbox{$\chi$}} }

\newcommand {\sgn} { {\mathrm {sgn}} }
\newcommand{\sgne}{{\mathrm{sgn}^{(\eps)}}}

\newcommand{\ds}{\displaystyle}
\newcommand{\ud}{\, \mathrm{d}}

\newcommand{\beq}{\begin{equation}}
\newcommand{\beqa}{\begin{eqnarray}}
\newcommand{\bea} {\begin{array}{ll}}
\newcommand{\beqan}{\begin{eqnarray*}}
\newcommand{\eeq}{\end{equation}}
\newcommand{\eeqa}{\end{eqnarray}}
\newcommand{\eeqan}{\end{eqnarray*}}
\newcommand{\eea} {\end{array}}
\newtheorem{theorem}{Theorem}[section]
\newtheorem{lemma}[theorem]{Lemma}

\newtheorem{remark}[theorem]{Remark}
\newtheorem{prop}[theorem]{Proposition}
\newtheorem{coro}[theorem]{Corollary}
\newcommand{\qed}{{ \hfill
                       {\unskip\kern 6pt\penalty 500
                       \raise -2pt\hbox{\vrule\vbox to 6pt{\hrule width 6pt
                       \vfill\hrule}\vrule} \par}  \medskip }}

\title{\Large \bf
 Two-dimensional pseudo-gravity model}
 \author[1]{Julien Barr\'e\thanks{ {\tt julien.barre@univ-orleans.fr}}}
  \author[2]{Dan Crisan\thanks{ {\tt d.crisan@imperial.ac.uk}}}
\author[3]{Thierry Goudon\thanks{ {\tt thierry.goudon@inria.fr}}}
 
\affil[1]{Laboratoire MAPMO, UMR CNRS 7349, Universit\'e d'Orl\'eans\\ 

\& 
 Institut Universitaire de France}
\affil[3]{Universit\'e C\^ote d'Azur, Inria,  CNRS, LJAD}
\affil[2]{Imperial College London, Dept. of Mathematics\\

   Huxley Building,
   180 Queens Gate,
   London SW7 2BZ, UK
}
\date{}
\begin{document}

\maketitle

\abstract{We analyze  a simple macroscopic model describing the evolution of a cloud of particles confined in a magneto-optical trap.
The behavior of the particles is mainly driven by self--consistent attractive forces. In contrast to the standard model of gravitational forces, 
the force field does not result from a potential;
moreover, the non linear coupling is more singular than the coupling based on the Poisson equation.
We establish the existence of solutions, under a suitable smallness condition on the total mass, or, equivalently, for a sufficiently large diffusion coefficient.
When a symmetry assumption is fulfilled, the solutions satisfy strengthened   
estimates (exponential moments).
We also investigate the convergence of the $N$-particles description towards the PDE system in the mean field regime.}

\medskip
\noindent
  \small{\bf{Key words. }}\small{Attractive forces.
  Convection--diffusion equations. Mean field regime. 
  }



\section{Introduction}

This  work is concerned with a simple mathematical model describing anisotropic magneto-optical traps (MOT).
In these devices, clouds of atoms are held together at very low temperatures through the action of well tuned lasers.
These lasers induce on each atom an external space dependent confining force, as well as a friction: these effects are responsible 
for the trapping and cooling of the atoms. The lasers also create effective interaction forces between the atoms. The precise description 
of these forces involves a full description of the laser field and its coupling with the atoms. The following simplification, while probably 
not always quantitatively accurate, is customary since the pioneering article \cite{Sesko90}: the interaction forces are divided into
\begin{itemize}
\item[i)] 
a repulsive force due to multiple diffusion of photons, 
which is usually approximated by a Coulomb force (predicted in \cite{Sesko90}) and
\item[ ii)] an attractive long-range force, the so--called "shadow effect"  (predicted in \cite{Dalibard88}), that bears 
some similarity with gravity, 
and is the main subject of this article. 
\end{itemize}
In a standard, roughly spherical, cloud, the repulsive force dominates. Nevertheless, if an external potential forces the cloud into a very elongated cigar shape, or a very thin pancake shape, the attractive force is expected to dominate, and the repulsive force may be neglected in a first approximation \cite{Barre14,Chalony13}. This is the regime we are interested in.

A  typical MOT involves $10^6$ to $10^{10}$ interacting particles. Although in experiment trapping the atoms in a pancake-shaped cloud would probably contain less atoms, 
it is then relevant to make use of a partial differential equations describing the particles' density, instead 
of considering the dynamics of the individual particles. A reasonable model may be a 3D non--linear Fokker-Planck, or a McKean-Vlasov, equation. 
However, in order to  describe the cigar- or pancake-shaped clouds observed in the experiments 
it makes sense to use a large scale approach, and to integrate over the 
 small dimension(s). After some approximations,  one is left with an effective 1D or 2D nonlinear partial differential equations (PDE). 
The 1D equation obtained this way coincides with the mean-field description a 1D damped self-gravitating system \cite{Chalony13} and is well-known. We thus concentrate on the 2D case.
The 2D nonlinear PDE studied here has its own interest, independently of
the relation with the MOT experiments: it bears some similarities with a 2D damped self-gravitating system (also known as the Smoluchowski model in astrophysics \cite{Chandra} or 
the Keller-Segel chemotactic model \cite{KellerSegel70, KellerSegel71}). 
 Therefore, a natural question is to determine whether or not  singularities appear in finite time, depending on certain thresholds, 
as this is the case for  the Keller--Segel model, see the review \cite{Horst1, Horst2}.

We are interested in the particle density $(x,y,t)\mapsto \rho(x,y,t)$, which is a scalar non--negative quantity that depends on the time $t\geq 0$ and space variables $(x,y)\in \R^2$.
Its evolution  is governed by the following non linear PDE 
\begin{equation}\label{pde}
\partial_t \rho =\nabla\cdot\left(D \nabla \rho -\vec{F}[\rho] \rho \right),
\end{equation}
where the constant  $D>0$ is given and  the self consistent force field \[
\vec{F}[\rho]=\begin{pmatrix}F_x[\rho]\\ F_y[\rho]\end{pmatrix},\]
 is defined by 
\begin{equation}\label{def_f}\begin{array}{l}
F_x[\rho](x,y,t) = -\ds\int \sgn(x-x') \rho(x',y,t)\ud x', \\
F_y[\rho](x,y,t) = -\ds\int \sgn(y-y') \rho(x,y',t)\ud y'.
\end{array}
\end{equation}
The problem is complemented with an initial data 
\begin{equation}\label{ci}
\rho\Big|_{t=0}=\rho_0.\end{equation}
Similar to the Keller--Segel model, the force is thus defined through a convolution formula. As a consequence of 
the fact that the (distributional) derivative of the function $x\mapsto \sgn(x)$ is $2\delta_0$, where $\delta_0$ is the Dirac delta distribution at $0$,
 we observe that (mind the sign)
\begin{equation}\label{div_f}
\nabla\cdot \vec{F}[\rho]=-4\rho\leq 0.\end{equation}
The divergence of the force field of the Keller-Segel system satisfies the same relation.
However, there are crucial differences with the Keller--Segel system that make the analysis here different:
\begin{itemize}
\item the force does not have the potential structure ($\vec{F}$ cannot be expressed as the gradient of a potential), and, accordingly, 
we cannot derive estimates related to the evolution of a potential energy,
\item the convolution acts only on a single direction variable; hence we cannot expect any regularisation effect similar to the one given 
by the coupling of the force through the Poisson equation,
\item we cannot use symmetry properties for expressing the 
force term 
$\iint \vec {F}[\rho]\rho\cdot\nabla\varphi \ud y\ud x$, for $\varphi\in C^\infty_c(\mathbb R^2)$,  in a convenient weak sense, which is a crucial ingredient
 when dealing with the Keller--Segel equation, see e.~g.~\cite{Po, PS, Sc}.  
\end{itemize}
We wish to investigate the  existence, uniqueness of a solution of \eqref{pde}--\eqref{ci} and to devise and analyze a particle method 
which can be used to perform simulation of the  PDE.
To be more specific, our strategy is as follows:
\begin{enumerate}
\item Introduce a regularized PDE
\begin{equation}\label{reg_pde}
\partial_t \rho^{(\eps)} =\nabla\cdot\left(\nabla \rho^{(\eps)} -\vec{F}^{(\eps)} [\rho^{(\eps)} ] \rho^{(\eps)}  \right)
\end{equation}
where the kernel  $\left(\sgn(x)\delta(y),\delta(x)\sgn(y)\right)$ in \eqref{def_f} is smoothed out. 
We take
\begin{equation}\label{approxF}
\begin{array}{lll}
F^{(\varepsilon)}_x[\rho](x,y,t) &=& -\ds\iint \sgn^{(\eps)}(x-x') \delta^{(\eps)}(y-y')\rho(x',y',t)\ud x'\ud y' , \\
F^{(\varepsilon)}_y[\rho](x,y,t) &=& -\ds\iint \sgn^{(\eps)}(y-y') \delta^{(\eps)}(x-x')\rho(x',y',t)\ud x'~\ud y' 
\end{array}\end{equation}
with
\begin{eqnarray}
\sgn^{(\eps)}(u) &=& 2 \frac{1}{\varepsilon\sqrt{2\pi}} \int_0^u e^{-\frac{v^2}{2\varepsilon^2}}\ud v, \nonumber \\
\delta^{(\eps)}(u) &=& \frac12 \frac{\ud}{\ud u} \sgn^{(\eps)}(u) =  \frac{1}{\varepsilon\sqrt{2\pi}}e^{-\frac{u^2}{2\varepsilon^2}}.\nonumber
\end{eqnarray}
Denoting by $\star_x$ (resp. $\star_y$)  the convolution with respect to the variable  $x$ (resp. the  variable $y$), we observe  that
\begin{eqnarray}
F^{(\varepsilon)}_x[\rho] &=& T^{(\eps)}( \sgn \star_x \rho) =\sgn \star_x (T^{(\eps)} \rho), \nonumber \\
F^{(\varepsilon)}_y[\rho] &=& T^{(\eps)}( \sgn \star_y \rho) = \sgn \star_y(T^{(\eps)} \rho),   \nonumber
\end{eqnarray}
where $T^{(\eps)}$ stands for the convolution with the normalized 2-d Gaussian kernel.\\
\item Establish a priori estimates that are uniform with respect to  $\varepsilon$. We 
obtain several such estimates, typically $L^p$ and moment estimates, based on dissipative properties of the equation, at the price of assuming the diffusion coefficient $D$ large enough. Section~\ref{sec:estimates} includes these estimates.
\item Show the existence and uniqueness of solutions $\rho^{(\eps)}$ of the regularized PDE \eqref{reg_pde}. To this end, we employ a suitable  fixed point approach, described in  Section~\ref{sec:existence}
\item Use the a priori estimates to prove global existence of the solution of the original equation, at least when $D$ is large enough. 
We present two proofs. The first relies on 
quite standard compactness arguments. As mentioned above the difficulty 
is related to the non--linear term $\vec{F}[\rho]\rho$ and the adopted functional framework 
should be constructed  so that the product makes sense and is stable.
The second approach is more precise and establishes directly that the sequence of approximated solutions $\big(\rho^{(\eps)}\big)_{\eps>0}$ 
 satisfies the Cauchy criterion in a certain norm. However this approach 
 requires certain symmetry assumptions and fast enough decay of the initial state. These additional assumptions  allow us to  derive 
exponential moments, and weighted estimates on the gradient of the unknown.
 This analysis is detailed in  Section~\ref{sec:convergence}. 
\item Introduce a stochastic  system of $N\gg 1$ particles, with a regularized interaction, and prove that the empirical measure  converges towards a solution of the  PDE when $N\to \infty$. Assuming that the  number of particles
$N $ is proportional to $e^{C/\eps^2}$ 
with $\eps$ being the regularizing parameter, one can obtain particle approximations that are arbitrarily close to $\rho$, on any fixed time interval. In particular we show that one can get an upper bound for the Wasserstein distance between the particle approximation and $\rho$ of order $\eps^{-\nu}$, where $\nu>0$ is a certain constant independent of $\eps$. Put it differently, we show that if $\eps$ is of order $(\log(N))^{-{1\over 2}}$, then the rate of convergence of the Wasserstein distance between the particle approximation and $\rho$ is also of logarithmic order, see Theorem~\ref{ratep}. The analysis is  presented in  Section~\ref{sec:particles}. 

\item Run numerical simulations using the particle representation obtained in Section 5 and compare it with the PDE method introduced in \cite{CCH}. In this way we illustrate the existence results covered by Theorems \ref{main} and \ref{main2}. As we will see, the constraint on the  diffusion coefficient (condition \eqref{DcritL2}) required for the two theorems to be valid is not optimal: the solution can apparently be global in time for other values too. We also illustrate the convergence for the particles approximation. The rate of convergence of the particle approximation as a function of $N$ seems to be much better than that suggested by Theorem \ref{ratep}. These are covered in Section~\ref{Sec:Num}.
\end{enumerate}

\section{A priori estimates}
\label{sec:estimates}

\subsection{Moments}
Let $k\in\N$, $k\geq 2$. We set 
\[
m_k(t)=\iint (|x|^k+|y|^k) \rho(x,y,t)\ud x\ud y
\]
Using integration by parts yields
\begin{eqnarray}
\frac{\ud m_k}{\ud t}(t) &=& Dk(k-1) m_{k-2} (t)
\nonumber
\\
&&
+k\iint \left(\sgn(x)|x|^{k-1} F_x [\rho]+\sgn(y)|y|^{k-1} F_y[\rho]\right) \rho(x,y,t)\ud x\ud y 
\nonumber\\
&=& Dk(k-1) m_{k-2}(t)
\nonumber\\
&& -k \iiint \sgn(x)|x|^{k-1} \sgn(x-x')\rho(x',y)\rho(x,y,t)\ud x\ud x'\ud y
\nonumber\\
&& - k \iiint \sgn(y)|y|^{k-1} \sgn(y-y')\rho(x',y)\rho(x,y,t)\ud x\ud x'\ud y.
\nonumber
\end{eqnarray}
By exchanging the r\^ole of $x$ and $x'$, we find
\[\begin{array}{l}
\ds \iiint \sgn(x)|x|^{k-1} \sgn(x-x')\rho(x',y)\rho(x,y,t)\ud x\ud x'\ud y
\\
\qquad=
\ds\frac{1}{2} \iint [\sgn(x)|x|^{k-1}-\sgn(x')|x'|^{k-1}] \sgn(x-x')\rho(x',y,t)\rho(x,y,t)\ud x\ud x'\ud y\geq 0
\end{array}\]
since $x\mapsto \sgn(x)|x|^{k-1}$ is non--decreasing.
A similar remark applies for the integral coming from $F_y$.
Therefore, the moments satisfy the following relation
\begin{equation}\label{mtk}
\frac{\ud m_k}{\ud t} \leq Dk(k-1) m_{k-2}.
\end{equation}
In particular, since the total mass is conserved
\[
\ds\frac{\ud}{\ud t} \ds\iint \rho(x,y,t)\ud y\ud x=0,\]
we obtain 
\[
m_2(t)\leq m_2(0) +2DM_0 t,
\]
with
\[M_0= \ds\iint \rho_0(x,y)\ud y\ud x.\]

\subsection{Entropies}
Let $h:[0,\infty)\to \R$ be a convex function and  write
\[
H[\rho] = \iint h ( \rho)\ud y\ud x.
\]
We have
\[
\begin{array}{lll}
\ds\frac{\ud H[\rho]}{\ud t} &=&\ds \iint h'( \rho) \nabla\cdot\left(D \nabla \rho -\vec{F}[\rho] \rho \right) \ud y\ud x\\
&=& -D\ds\iint |\nabla \rho|^2 h''(\rho)\ud y\ud x +\iint \nabla \rho \cdot \vec{F}[\rho]\  \rho h''(\rho)\ud y\ud x.
\end{array}\]
Let  $q$ be an anti-derivative of 
\[q'(\rho)=\rho h''(\rho).\]
We thus arrive at 
\begin{equation}
\begin{array}{lll}
\ds\frac{\ud H[\rho]}{\ud t} &=& -D\ds\iint |\nabla \rho|^2 h''(\rho)\ud y\ud x 
- \iint  q(\rho)\nabla  \cdot \vec{F}[\rho] \ud y\ud x\\
\\
&=& -D\ds\iint |\nabla \rho|^2 h''(\rho)\ud y\ud x 
+4 \iint \rho q(\rho)\ud y\ud x
\label{eq:dH}
\end{array}
\end{equation}
by virtue of \eqref{div_f}.
In order to compensate the non--linearity 
in the last integral by the dissipated term, we can make use of the following 
 Gagliardo--Nirenberg--Sobolev inequality
 (see e.~g.~\cite[p.~125]{LN} or \cite[Th. IX.9 with eq. (17) \& eq. (85) p.~195]{Brez}), 
 which holds in $\mathbb R^2$  for any $p\geq 1$:
 \begin{equation}\label{GNS}
\iint \xi^{p+1}\ud y\ud x\leq C_p\iint \xi\ud y \ud x\times \iint |\nabla(\xi^{p/2})|^2\ud y \ud x.\end{equation}
Let us detail how the estimates work in different cases:
\begin{itemize}
\item Entropy $h(z)=z\ln(z)$.

\noindent We get $zq(z)=z^2$ and we use \eqref{GNS} with $p=1$.
Remarking that $\frac{|\nabla\rho|^2}{\rho}=4|\nabla\sqrt\rho|^2$ and  taking into account the mass conservation, we are led to 
\begin{equation}\label{entropy}
\ds\frac{\ud }{\ud t}\ds\iint \rho\ln (\rho)\ud y \ud x
+ 4(D -C_1M_0)\ds\iint |\nabla\sqrt\rho|^2\ud y\ud x \leq 0.
\end{equation}
It indicates a dissipation property when
the diffusion coefficient is large enough
\[D\geq C_1M_0.
\]
Based on this, we can  conjecture that solutions exist globally for large diffusion constants $D$.
\item $L^p$ estimates: $h(z)=z^p$.

\noindent
We get $zq(z)=(p-1)z^{p+1}$, and we use \eqref{GNS} with $p>1$.
Remarking that $h''(\rho)|\nabla\rho|^2=\frac4p (p-1)|\nabla\rho^{p/2}|^2$, we are led to 
\begin{equation}\label{Lpdiss}
\ds\frac{\ud }{\ud t}\ds\iint \rho^p \ud y \ud x
+ 4(p-1)\left(\ds\frac{D}{p} -C_pM_0\right)\ds\iint |\nabla\rho^{p/2}|^2\ud y\ud x \leq 0.
\end{equation}
Eq.~\eqref{Lpdiss} shows that the $L^p$ norm of the solution is a non-increasing function of time when  $D$ is large enough, but 
how large depends on $p$ with this approach. We are going to obtain finer estimates for large values of the constant $p$.
\end{itemize}

In order to eliminate the too restrictive condition on $D$, we use a different approach for the $L^p$ estimate. To this end, we use the 
Cauchy--Schwarz inequality and \eqref{GNS} and we obtain
\[\begin{array}{lll}\ds
\iint \rho^{p+1} \ud y \ud x
&\leq&\ds \left(\iint \rho^2 \ud y \ud x\right)^{1/2} \left(\iint \rho^{2p} \ud y \ud x\right)^{1/2}
\\
& \leq &\ds
\left(\iint \rho^2 \ud y \ud x\right)^{1/2} \left( C_1\ds\iint \rho^p\ud y\ud x
\ds\iint |\nabla\sqrt{\rho^p}|\ud y\ud x\right)^{1/2}.
\end{array}\]
Going back to \eqref{eq:dH}, still with $h(z)=z^p$, 
the elementary inequality  $AB\leq \alpha \frac{A^2}{2}+\frac{B^2}{2\alpha}$
with an appropriate choice of $\alpha>0$ leads us to 
\[
\ds\frac{\ud}{\ud t} \iint \rho^p \ud y \ud x+2D\ds\frac{p-1}{p}\iint |\nabla \rho^{p/2}|^2
\ud y\ud x
\leq \ds\frac{ C_1^2p}{8D(p-1)}\times 16(p-1)^2\iint \rho^2\ud y\ud x\iint \rho^p\ud y\ud x
.\]  
From now on, we assume that 
\begin{equation}\label{DcritL2}
D>2C_2M_0.\end{equation}
Accordingly, the $L^2$ norm is dissipated and 
$$\ds\iint \rho^2(t)\ud y\ud x\leq \ds\iint \rho^2_0\ud y\ud x$$
holds.
Therefore, we arrive at 
\begin{equation}\label{estLp00}
\ds\frac{\ud}{\ud t} \iint \rho^p \ud y \ud x+2D\ds\frac{p-1}{p}\iint |\nabla \rho^{p/2}|^2
\ud y\ud x
\leq K_1  \ p^2\iint \rho^p\ud y\ud x
\end{equation}
with $K_1=\frac{2C_1^2\|\rho_0\|^2_{L^2}}{D}$.

We use this relation to derive a $L^\infty$ estimate, through an iterative argument on the exponent $p$ which dates back to  
 \cite{Alikakos}.
 Let us set 
 \[p_k=2^k,\qquad v_k=\rho^{p_k}.\]
 Let $\omega>0$. Eq. \eqref{estLp00} tells us that 
\[ e^{-\omega t}\ds\frac{\ud}{\ud t} \left(e^{\omega t} \iint v_k \ud y \ud x\right)
  +2D\ds\frac{p_k-1}{p_k}\iint |\nabla v_{k-1}|^2
\ud y\ud x
\leq 
  (K_1p_k^2 +\omega)  \iint v_k\ud y\ud x
.\]
We are going to estimate the right hand side by using the following 
Gagliardo--Nirenberg--Sobolev inequality
 (see e.~g.~\cite[p.~125]{LN} or \cite[eq. (85) p.~195]{Brez})
 \begin{equation}\label{GNS2}
\iint \xi^{2}\ud y\ud x\leq \bar C_2\iint \xi\ud y\ud x\left(\iint |\nabla\xi|^2\ud y\ud x\right)^{1/2}.\end{equation}
We combine this information with the Young inequality 
as follows 
\[
\iint \xi^{2}\ud y\ud x\leq \ds\frac{\bar C_2\omega}{2}\iint |\nabla\xi|^2\ud y\ud x
+  \ds\frac{\bar C_2}{2\omega} \left(\iint \xi\ud y\ud x
\right)^{2}.\]
We choose $\omega=\omega_k>0$ small enough to ensure 
\[
\bar C_2(K_1 p_k^2 +\omega_k)\ds\frac{\omega_k}{ 2}\leq D\ds\frac{p_k-1}{p_k}.
\]
Since $v_k=v_{k-1}^2$, we are thus led to 
\[ e^{-\omega_k t}\ds\frac{\ud}{\ud t} \left(e^{\omega_k t} \iint v_k \ud y \ud x\right)
  +D\ds\frac{p_k-1}{p_k}\iint |\nabla v_{k-1}|^2
\ud y\ud x
\leq 
  \ds\frac{\bar C_2(K_1p_k^2 +\omega_k) }{2\omega_k}\left(  \iint v_{k-1}\ud y\ud x
\right)^2.\]
Integrating from this relation, we obtain
\[
\begin{array}{lll}
\ds\iint v_k(t) \ud y \ud x
&\leq& e^{-\omega_k t}\left(\ds \iint v_k(0) \ud y \ud x \right.
\\ &&\qquad\qquad\qquad \left. +
\ds\int_0^t e^{-\omega_k s}\ds\frac{\bar C_2(K_1p_k^2 +\omega_k) }{2\omega_k}\left(  \iint v_{k-1}(s)\ud y\ud x\right)^2\ud s\right)
\\
&\leq & e^{-\omega_k t}\ds\iint v_k(0) \ud y \ud x
\\ &&  +e^{-\omega_k t}\ds\frac{e^{\omega_k t}-1}{\omega_k} \ \ds\frac{\bar C_2(K_1p_k^2 +\omega_k) }{2\omega_k}
\ds\sup_{0\leq s\leq t} \left(  \iint v_{k-1}(s)\ud y\ud x\right)^2.
\end{array}\]
In the right hand side, we make a convex combination appear, and we infer that
\[
 \iint v_k(t) \ud y \ud x
\leq \max\left\{\iint v_k(0) \ud y \ud x, \ds\frac{\bar C_2(K_1p_k^2 +\omega_k) }{2\omega_k^2}\left(\ds\sup_{0\leq s\leq t}  \iint v_{k-1}(s)\ud y\ud x
\right)^2\right\}.
\]
Let us set 
\[
L=\max\big(\|\rho_0\|_{L^1},\|\rho_0\|_{L^\infty}\big)
,
\qquad
\delta_k=\ds\frac{\bar C_2(K_1p_k^2 +\omega_k) }{2\omega_k^2}
.\]
Note that $\omega_k$ behaves like $\frac{1}{p_k^2}$, and thus we can dominate $\delta_k\leq Mp_k^6$ for some $M>0$, so that, finally, we can find $A>0$ such that $
\delta_k\leq A^k$.
A direct recursion shows that
$$
 \iint v_k(t) \ud y \ud x\leq 
 \delta_k \delta_{k-1}^{p_1}\ldots\delta_{1}^{p_{k-1}}L^{p_k}
$$
which implies 
\[
\|\rho(t)\|_{L^{p_k}} \leq L \left( A^{r_k}\right)^{1/p_k},\qquad 
r_k=\ds\sum_{\ell=0}^{k} (k-\ell)p_\ell.
\]
Since 
\[
\ds\frac{r_k}{p_k}=\ds\frac12\ds\sum_{j=1}^k j\left(\ds\frac12\right)^{j-1}=\ds\frac12 \ds\frac{\ud}{\ud s}\left(\ds\frac{1-s^{k+1}}{1-s}\right)\Big|_{s=1/2}=
2\Big(
1+(k+2)e^{-(k+1)\ln(2)}
\Big)
\]
admits a finite limit as $k\to \infty$, we deduce that the sequence $\big(\|\rho(t)\|_{L^{p_k}}\big)_{k\in\N}$ is bounded . The $L^\infty$ bound follows by letting $k$ go to $\infty$, and the bound depends on the initial $L^1$ and 
$L^\infty$ norms.
The minimal $D$ needed for this bound to be valid is unknown.
We can recap our findings as follows.

\begin{prop}\label{apriori}
Let $\rho$ be a sufficiently smooth 
solution of  \eqref{pde}--\eqref{ci}.
Then, $\rho$ satisfies the following properties:
\begin{itemize}
\item[i)] mass is conserved $\iint \rho(t)\ud y\ud x=\iint \rho_0\ud y\ud x=M_0$,
\item[ii)] if $\rho_0\in L^p(\R^2)$ and $D>pC_pM_0$,\footnote{The constant $C_p$ is the constant appearing in the Gagliardo--Nirenberg--Sobolev inequality \eqref{GNS}. } then, $\|\rho(t)\|_{L^p}\leq \|\rho_0\|_{L^p}$,
\item[iii)]  if $\rho_0\in L^1\cap L^\infty (\R^2)$ and $D>2C_2M_0$, then there exists a constant $M>0$ such that 
$0\leq \rho(y,x,t)\leq M$ holds for a.e.~$t\geq 0$, $(x,y)\in\R^2$.
\item [iv)] if $(x,y)\mapsto (x^2+y^2)\rho_0(x,y)\in L^1(\R^2)$, then, for any $t\geq 0$, 
$(x,y)\mapsto (x^2+y^2)\rho(x,y,t)\in L^1(\R^2)$, and $m_2(t)\leq m_2(0) +2DM_0t$.
\end{itemize}
\end{prop}

\subsection{Estimates for the regularized problem} 

To analyze the solutions of the regularized PDE \eqref{reg_pde}-\eqref{approxF} and justify their convergence as $\eps$ tends to $0$, we will need estimates uniform with respect to  $\eps$. The following proposition is the equivalent of Proposition~\ref{apriori} for the regularized solution.

\begin{prop} \label{lemma:rho_eps_estimates}
Let  $\big(\rho^{(\varepsilon)}\big)_{\eps>0}$ be the sequence of  solutions of the regularized PDE \eqref{reg_pde}--\eqref{approxF}, associated to the  initial data $\big(\rho^{(\varepsilon)}_0\big)_{\eps>0}$.
We assume that 
\[\textrm{$\big(\rho^{(\varepsilon)}_0\big)_{\eps>0}$ is bounded in $L^1(\R^2)\cap L^\infty(\R^2)$.}\]
Then,  the following properties are satisfied:
\begin{itemize}
\item[i)] mass is conserved $\iint \rho^{(\varepsilon)}(t)\ud y\ud x=\iint \rho^{(\varepsilon)}_0\ud y\ud x$,
\item[ii)] if  $D>pC_pM_0$, then $\rho^{(\varepsilon)} $ is bounded in  $L^\infty(0,\infty;L^p(\R^2))$ and  $\|\rho^{(\varepsilon)} (t)\| _{L^p}\leq \|\rho^{(\varepsilon)}_0\|_{L^p}$,
\item[iii)]  if  $D>2C_2M_0$, then $\rho^{(\varepsilon)} $ is bounded in $L^\infty(0,\infty;L^2(\R^2))\cap L^\infty((0,\infty)\times\R^2)$,
\item [iv)] if $(x,y)\mapsto(x^2+y^2)\rho^{(\varepsilon)}_0(x,y)$ is bounded in $ L^1(\R^2)$, then 
$(x^2+y^2)\rho^{(\varepsilon)}(x,y,t)$ is bounded in  $L^\infty(0,T;L^1(\R^2))$ for any $0<T<\infty$.
\end{itemize}
\end{prop}

\noindent
{\bf Proof.}
Item i) is clear. 
The proof of iv) repeats the same arguments as above, with a direct comparison to a pure diffusion.
For ii) and iii), we will need the following consequence of the definition \eqref{approxF}
\[
\partial_x F_x^{(\varepsilon)}[\rho^{(\varepsilon)}](x,y,t) = -2 \iint \delta^{(\varepsilon)}(x-x')\delta^{(\varepsilon)}(y-y')\rho^{(\varepsilon)}(x-x',y-y',t)\ud x'\ud y'
\]
so that \eqref{div_f} becomes
\[
\nabla \cdot F^{(\varepsilon)} [\rho^{(\varepsilon)}]=-4 T^{(\varepsilon)}( \rho^{(\varepsilon)})
\]
where, as said above,  $T^{(\varepsilon)}$ is the convolution operator with the normalized 2d Gaussian kernel. Furthermore, the H\"older inequality yields
\[\begin{array}{lll}\ds
\iint (\rho^{(\varepsilon)})^p  T^{(\varepsilon)}( \rho^{(\varepsilon)})\ud y\ud x
&\leq & \ds\left(\iint (\rho^{(\varepsilon)})^{p+1} \ud y\ud x\right)^{p/(p+1)} \left(\iint |T^{(\varepsilon)} ( \rho^{(\varepsilon)})|^{p+1} \ud y\ud x\right)^{1/(p+1)} \\
&\leq & \ds\left(\iint (\rho^{(\varepsilon)})^{p+1}\ud y\ud x\right)^{p/(p+1)} \left(\iint (\rho^{(\varepsilon)})^{p+1}\ud y\ud x\right)^{1/(p+1)} \\
&\leq & \ds\iint (\rho^{(\varepsilon)})^{p+1} \ud y\ud x.\label{eq:divepsi}
\end{array}\]
With this observation, we can go back to 
 \eqref{eq:dH} adapted to the regularized problem and we derive the estimates as we did for the singular equation. 
 We refer the reader to \cite{Blanchet06} for similar reasonings.
 \qed

\section{Regularized problem}
\label{sec:existence}

Let $\eps>0$.  The initial data  $\rho^{(\eps)}_0$ is a given non--negative  function in $L^1(\R^2)\cap L^\infty(\R^2)$.
We  introduce the operator $$\mathcal{T}:g\mapsto \mathcal T(g)=\rho$$ where $\rho$ is the solution of
the linear parabolic PDE
\begin{equation}\label{linear}
\partial_t \rho =D\Delta \rho -\nabla \cdot (\vec{F}^{(\eps)}[g] \rho),\qquad
\rho\Big|_{t=0}=\rho^{(\eps)}_0.
\end{equation}
We will show that $\mathcal T$ fulfils the hypotheses
of the Schauder theorem in a suitable functional framework. This will lead to the existence of a fixed point, which defines a solution of the non--linear problem.
Then, we will investigate the uniqueness independently.
Gathering together these arguments, we will prove the following statement.

\begin{theorem}
Let $\rho^{(\eps)}_0\in L^1(\R^2)\cap L^\infty(\R^2)$ be a non-negative function.
Then, the problem \eqref{reg_pde}--\eqref{approxF} with $\rho\big|_{t=0}=\rho^{(\eps)}_0$ admits a unique solution 
$\rho^{(\eps)}\in C([0,T];L^2(\R^2))\cap L^2(0,T;H^1(\R^2))$.
\end{theorem}

\subsection{Preparing for the Schauder theorem: a priori estimates}
We observe  that $$
\begin{array}{l}
\sgn^{(\eps)} \in C^\infty(\R) \cap L^\infty(\R), 
\\
x\mapsto \ds\frac{\sgn^{(\eps)}(x)}{\sqrt{1+x^2}} \in L^2(\mathbb{R}), \\
\delta^{(\eps)}=\ds\frac12 \ds\frac{\ud }{\ud x}\sgn^{(\eps)} \in C^\infty(\R)\cap L^1(\R)\cap  L^\infty(\mathbb{R}).
\end{array}$$
Owing
to these properties, we obtain 
estimates (that depend on $\eps$) on the regularized force.
\begin{lemma}\label{lemma:estF}
The following estimates hold
\begin{itemize}
\item[i)] $|F^{(\eps)} _x[g]|\leq \|\sgne\|_\infty \|\delta^{(\eps)}\|_\infty \|g\|_{L^1} = \frac{1}{\eps\sqrt{2\pi}}\|g\|_{L^1}$,
\item[ii)] $|F^{(\eps)} _x[g]| \leq \sqrt\pi\ \|\delta^{(\eps)}\|_{L^2}  
\left( \int(1+x'^2)g^2(x',y')\ud x'\ud y'\right)^{1/2}$,
\item[iii)] $|\partial_x F^{(\eps)} _x[g]| \leq 2\|\delta^{(\eps)}\|^2_{L^1} 
\|g\|_{L^\infty} = 2 \|g\|_{L^\infty}$,
\item[iv)] $|\partial_x F^{(\eps)} _x[g]| \leq 2\|\delta^{(\eps)}\|^2_{L^\infty} 
\|g\|_{L^1} = \frac{2}{\pi \eps^2}\|g\|_{L^1} $.
\end{itemize}
Of course, the same estimates apply to $F^{(\eps)} _y$ as well.
\end{lemma}

\noindent{\bf Proof.}
It is worth bearing in mind that 
\[
0\leq \delta^{(\eps)}(x)\leq \ds\frac{1}{\eps\sqrt{2\pi}},\qquad
|\sgne(x)|\leq 1.\]
Items i), iii) and iv) are direct consequences of estimates on convolution products.
For ii) we use  the Cauchy-Schwarz inequality twice to obtain
\[\begin{array}{l}
|F^{(\eps)} _x[g](x,y)| \leq \ds \int \delta^{(\eps)}(y-y') \left( \int\frac{|\sgne(x-x')|^2}{1+x'^2}\ud x'\right)^{1/2}\left(\int (1+x'^2)g^2(x',y')\ud x'\right)^{1/2} \ud y'  \\
\quad\leq \ds \left( \int |\delta^{(\eps)}|^2(y-y')\ud y'\right)^{1/2} \left(\int \frac{|\sgne(x-x')|^2}{1+x'^2}\ud x'\right)^{1/2} \left(\int(1+x'^2)g^2(x',y')\ud x'\ud y'\right)^{1/2} \\
\quad\leq \ds \|\delta^{(\eps)}\|_{L^2}
\times \sqrt{\pi} \times \left(\int(1+x'^2)g^2(x',y')\ud x'\ud y'\right)^{1/2} .
\end{array}\]
\qed

\noindent
For any $g\in L^\infty(0,\infty;L^1(\R^2))$, owing to the observations in Lemma~\ref{lemma:estF}, 
the linear problem \eqref{linear} admits a unique solution, say in $C([0,\infty];L^2(\R^2))\cap L^2(0,\infty;H^1(\R^2))$, see \cite[Th.~X.9]{Brez}.
We can now derive estimates on the solution of \eqref{linear}.
\begin{lemma}\label{lemma:estrho}
Let $\rho=\mathcal T(g)$ be the solution of \eqref{linear}. It satisfies
\begin{itemize}
\item[i)] For any fixed time $t>0$ and any $p\in [1,\infty]$, $\rho(t) \in L^p(\R^2)$. More precisely, we have
\[
\iint \rho(x,y,t)^p\ud y\ud x \leq e^{4 (p-1)t\|g\|_{L^\infty(0,\infty;L^1(\R^2))} \|\delta^{(\eps)}\|_{L^\infty}^2
} \iint \rho^{(\eps)}_0(x,y)^p\ud y\ud x,
\]
and $\|\rho(t)\|_{L^\infty}\leq e^{4t \|g\|_{L^\infty(0,\infty;L^1(\R^2))} \|\delta^{(\eps)}\|_{L^\infty}^2 } \|\rho^{(\eps)}_0\|_{L^\infty}$.
\item[ii)] For any fixed time $t$, $\iint (x^2+y^2) \rho(x,y,t)\ud y\ud x$ is finite. More precisely, we have
\[\begin{array}{l}
\ds\iint (x^2+y^2) \rho(x,y,t)\ud y\ud x\\
\ds \leq e^t \left(
\iint (x^2+y^2) \rho^{(\eps)}_0(x,y)\ud y\ud x
+t\|\rho^{(\eps)}_0\|_{L^1}\Big(4D+\|\sgne\|_{L^\infty}^2\|\delta^{(\eps)}\|_\infty^2\|g\|_{L^\infty(0,\infty;L^1(\R^2))}^2\Big) \right).
\end{array}\]
\item[iii)]For any $0\leq t\leq T<\infty$, we have 
$$\ds\|\nabla  \rho\|^2_{L^2((0,t)\times\R^2)}\leq \frac{1}{2D} 
 e^{4T \|g\|_{L^\infty(0,\infty;L^1(\R^2))} \|\delta^{(\eps)}\|_{L^\infty}^2 }\|\rho^{(\eps)}_0\|_{L^2}
.$$
\end{itemize}
\end{lemma}

\noindent
{\bf Proof.}
i) We compute
\[\begin{array}{l}
\ds\frac{\ud}{\ud t}\iint \rho^p \ud y\ud x+Dp(p-1)\iint \rho^{p-2}|\nabla \rho|^2 \ud y\ud x= \iint F^{(\eps)}[g]\rho p(p-1)\cdot \nabla \rho \rho^{p-2} \ud y\ud x  \\
\qquad
\qquad\qquad \ds 
= -(p-1)\iint \nabla\cdot F^{(\eps)}[g] \rho^p\ud y\ud x \\
\qquad
\qquad\qquad
\ds
\leq (p-1)\|\nabla\cdot F^{(\eps)}[g] \|_{L^\infty} \ds\iint \rho^p\ud y\ud x \\
\qquad
\qquad\qquad
\ds\leq 4\|\delta^{(\eps)}\|^2_{L^\infty} \|g\|_{L^1}\times (p-1)\iint \rho^p\ud y\ud x.\label{eq:Lpestimate}
\end{array}\]
The last line uses Lemma \ref{lemma:estF}-iv). Gr\"onwall's lemma then 
yields i).
The $L^\infty$ estimate follows by taking the limit $p\to \infty$.
Estimate iii) is obtained by specifying to the case $p=2$ and considering the dissipation term.

 ii. Let us use the shorthand notation $z=(x,y)$. We get 
\begin{eqnarray}
\frac{\ud}{\ud t}\iint |z|^2 \rho\ud z &=& 4D\iint \rho \ud z+2\iint \rho F^{(\eps)}[g] \cdot z \ud z \nonumber \\
&\leq& 4D\iint \rho \ud z+ 2\|F^{(\eps)}[g] \|_{L^\infty} \left(\iint |z|^2\rho \ud z \right)^{1/2}\left(\iint \rho \ud z \right)^{1/2} \nonumber \\
&\leq & 4D\iint \rho \ud z+ \|F^{(\eps)}[g] \|_{L^\infty}^2 \iint \rho \ud z+ \iint |z|^2\rho \ud z\nonumber \\
&\leq & (4D+ \|\sgne\|_\infty^2 \|\delta^{(\eps)}\|_\infty^2 \|g\|_{L^1}^2)\iint \rho \ud z+ \iint |z|^2\rho \ud z \nonumber
\end{eqnarray}
by using  Lemma \ref{lemma:estF}-i). The Gr\"onwall lemma allows us to conclude.

iii) We have\begin{eqnarray}
\ds\frac{\ud}{\ud t}\iint \rho^2 \ud y\ud x+2D\iint |\nabla \rho|^2 \ud y\ud x   &\leq& 4\|\delta^{(\eps)}\|^2_{L^\infty} \|g\|_{L^1}\times \iint \rho^2\ud y\ud x.
\end{eqnarray}
Inserting the estimate of item i) leads to
\begin{eqnarray}
\ds\frac{\ud}{\ud t}\iint \rho^2 \ud y\ud x+2D\iint |\nabla \rho|^2 \ud y\ud x   &\leq& C e^{Ct} \iint \rho^{(\eps)}_0(x,y)^2\ud y\ud x,
\end{eqnarray}
with $C=4\|\delta^{(\eps)}\|^2_{L^\infty}  \|g\|_{L^\infty(0,\infty;L^1(\R^2))}$.
Integrating over time then yields
\begin{eqnarray}
\left(\iint \rho^2 \ud y\ud x\right) (t)-\left(\iint \rho^2 \ud y\ud x\right) (t=0)+2D \ds\|\nabla  \rho\|^2_{L^2((0,t)\times\R^2)} &\leq& (e^{Ct}-1) \|\rho^{(\eps)}_0\|_{L^2}^2.
\end{eqnarray}
Hence
\begin{eqnarray}
\ds\|\nabla  \rho\|^2_{L^2((0,t)\times\R^2)} &\leq& \frac{e^{Ct}}{2D} \|\rho^{(\eps)}_0\|_{L^2}^2 \\
&\leq& \frac{1}{2D} 
 e^{4T \|g\|_{L^\infty(0,\infty;L^1(\R^2))} \|\delta^{(\eps)}\|_{L^\infty}^2 } \|\rho^{(\eps)}_0\|_{L^2}^2.
\end{eqnarray}
\qed

\subsection{Preparing for the Schauder theorem: definition of the functional framework}
Let $0<T<\infty$ be fixed once for all. We introduce  the  set
$\mathcal{C} $ consisting of the functions $g:[0,T]\times \mathbb{R}^2 \to [0,\infty)$, such that
\begin{itemize}
\item [i)] $\iint g\ud z =\iint \rho_0^{(\eps)}\ud z\leq M_0$,
\item [ii)] $\iint |z|^2 g \ud z\leq e^{C_1T}(\iint |z|^2 \rho^{(\eps)}_0\ud z  +C_2T)$,
\item[iii)] $\|g\|_\infty \leq e^{C_3T} \|\rho^{(\eps)}_0\|_\infty$,
\end{itemize}
By using the mass conservation property, the estimates in Lemma~\ref{lemma:estrho} allow us to choose the constants $C_1$, $C_2$ and $C_3$ (which depend on $\eps$) such  that 
$\mathcal{C}$ is convex, and stable upon application of $\mathcal{T}$.

\subsection{Preparing for the Schauder theorem: $\mathcal{T}$ is continuous}
We wish to establish the continuity of $\mathcal T:\mathcal{C} \rightarrow \mathcal{C} $ for the norm of $L^2((0,T)\times \R^2)$. 
For $i\in\{1,2\}$, let $\rho_i=\mathcal{T}(g_i)$, with $g_i\in \mathcal C$. 
By Lemma~\ref{lemma:estrho}-i), we already know that $\rho_i$ belongs to $L^\infty(0,T;L^2(\R^2))$.
We denote $\vec{F}_i^{(\eps)}=\vec F^{(\eps)}[g_i]$.
We get
\begin{eqnarray}
\frac{\ud}{\ud t} \iint (\rho_2-\rho_1)^2\ud z &=& -2D \iint |\nabla(\rho_2-\rho_1)|^2\ud z
+2 \iint \nabla(\rho_2-\rho_1)\cdot \vec{F}^{(\eps)}_2(\rho_2-\rho_1)\ud z
 \nonumber
\\
&&+
2\iint \rho_1\nabla(\rho_2-\rho_1)\cdot (\vec F^{(\eps)}_2-\vec F^{(\eps)}_1) \ud z
\nonumber \\
&\leq &  -2D \iint |\nabla(\rho_2-\rho_1)|^2 \ud z
+ \int |\nabla \cdot \vec F^{(\eps)}_2|\ (\rho_2-\rho_1)^2 \ud z
 \nonumber
\\ && +D\iint |\nabla(\rho_2-\rho_1)|^2 \ud z
+\frac{1}{D}\int \rho_1^2(\vec F^{(\eps)}_2-\vec F^{(\eps)}_1)^2 \ud z \nonumber \\
&\leq &-D\iint |\nabla(\rho_2-\rho_1)|^2 \ud z
+ 4 \|\delta^{(\eps)}\|^2_{L^\infty} 
\|g_2\|_{L^1} \iint (\rho_2-\rho_1)^2\ud z \nonumber \\
&&+ \frac{1}{D} \|\vec F^{(\eps)}_2-\vec F^{(\eps)}_1\|_{L^\infty}^2  \iint \rho_1^2\ud z .
\label{eq:L2estim}
\end{eqnarray}
We aim at controlling  $\|\vec F^{(\eps)}_2-\vec F^{(\eps)}_1\|_{L^\infty}^2$ by the difference  $g_2-g_1$ in $L^2$ norm. 
This cannot be done directly, and we should use further moment estimates.
To be more specific, we will use a splitting that makes $\|g_2-g_1\|_{L^2}$ appear plus an arbitrarily small contribution.
To this end, we use Lemma~\ref{lemma:estF}-ii).
For any $1<s<2$ and any $R>0$, we  write
\[\begin{array}{l}
\ds(F^{(\eps)}_{2,x}-F^{(\eps)}_{1,x})^2(x,y,t) 
\leq  \|\delta^{(\eps)}\|_{L^2}^2 \left(\int \frac{|\sgne(x-x')|^2}{1+|x'|^s}\ud x'\right) \iint (1+|x'|^s)(g_2-g_1)^2(x',y',t)\ud x'\ud y' \nonumber \\
\qquad\ds
\leq  C^{(\eps)} \iint (1+|z|^s) (g_2-g_1)^2(z,t) \ud  z
\nonumber \\
\qquad\ds
 \leq  C^{(\eps)}\left(\iint (g_2-g_1)^2(z,t)\ud z+  \iint_{|z|\leq R} |z|^s (g_2-g_1)^2(z,t)\ud z 
\right.\nonumber 
\\
\left.\qquad\qquad\qquad\qquad\qquad\qquad\qquad\qquad\qquad\qquad
+ \ds \iint_{|z|> R} |z|^s (g_2-g_1)^2(z,t)\ud z\right) 
\nonumber \\
\qquad\ds \leq C^{(\eps)}(1+R^s)\iint(g_2-g_1)^2(z,t)\ud z + C^{(\eps)} \iint_{|z|> R} \frac{|z|^2}{|z|^{2-s}} (g_2-g_1)^2 (z,t)\ud z
\nonumber \\
\qquad \ds\leq C^{(\eps)}(1+R^s)\|(g_2-g_1)(t)\|^2_{L^2}  + C^{(\eps)} \frac{\|g_2\|_{L^\infty}+\|g_1\|_{L^\infty}}{R^{2-s}} \iint |z|^2|g_2-g_1|(z,t)\ud z
 \nonumber \\
\qquad\ds \leq C^{(\eps)}(1+R^s)\|(g_2-g_1)(t)\|^2_{L^2}   + C^{(\eps)}  \frac{\|g_2\|_{L^\infty}+\|g_1\|_{L^\infty}}{R^{2-s}} 
\ds\left(\int |z|^2g_2 \ud z+\int |z|^2g_1\ud z \right) 
.\label{eq:FL2estim}
\end{array}\]
The same
inequalities obviously hold for $F^{(\eps)}_{2,y}-F^{(\eps)}_{1,y}$.
Coming back to  \eqref{eq:L2estim} yields
\[\begin{array}{l}
\ds\frac{\ud}{\ud t} \iint (\rho_2-\rho_1)^2 \ud z
+ D\iint |\nabla(\rho_2-\rho_1)|^2\ud z \\
\qquad \ds\leq   4 \|\delta^{(\eps)}\|^2_{L^\infty} \|g_2\|_{L^1} \iint (\rho_2-\rho_1)^2\ud z
 \\
\qquad\qquad+  \ds\frac{2C^{(\eps)}\|\rho_1(t)\|^2_{L^2}}{D}\left(
(1+R^s)\|(g_2-g_1)(t)\|^2_{L^2}+ 
 \frac{\|g_2\|_{L^\infty}+\|g_1\|_{L^\infty}}{R^{2-s}} 
\ds\left(\int |z|^2g_2 +\int |z|^2g_1 \right) 
\right).
%
%
%
%
\end{array}\]
Bearing in mind that $C_3=4\|\delta^{(\eps)}\|_\infty^2 M_0$, we are ready to use the Gr\"onwall lemma which leads to 
\begin{equation}\begin{array}{l}
\ds\iint (\rho_1-\rho_2)^2 (z,t)\ud z 
\\
\ds
\qquad \leq e^{C_3T}\left\{ \int(\rho_2-\rho_1)^2(z,0)\ud z \right.\\
\qquad\qquad \ds+
\frac{2C^{(\eps)}\|\rho_1\|^2_{L^\infty(0,T;L^2(\R^2))}(1+R^s)}{D} \ds\int_0^t \iint (g_2-g_1)^2(z,\tau) \ud z\ud \tau
 \nonumber \\
\qquad\qquad \ds\left. + \frac{2C^{(\eps)}\|\rho_1\|^2_{L^\infty(0,T;L^2(\R^2))}(\|g_2\|_{L^\infty}+\|g_1\|_{L^\infty})}{DR^{2-s}}
\ds\int_0^t \iint |z|^2(g_2 +g_1)(z,\tau)\ud z\ud \tau    \right\}.
\label{eq:main_estimate}
\end{array}\end{equation}
When  $\rho_1$ and $\rho_2$ have the same initial condition  the first term of the right hand side vanishes.

Take  $g \in \mathcal{C}$ and consider a sequence $\big(g_n\big)_{n\in\mathbb{N}}\in\mathcal{C}$, such that $g_n \to g$ in $L^2([0,T]\times \mathbb{R}^2)$. 
We apply 
\eqref{eq:main_estimate} with 
 $\rho_n=\mathcal{T}(g_n)$ and
$\rho=\mathcal{T}(g)$; it reads
\[\begin{array}{l}
\ds\iint (\rho_n-\rho)^2 (z,t)\ud z \\
\qquad
\leq
\ds
e^{C_3T}\frac{2C^{(\eps)}\|\rho\|^2_{L^\infty(0,T;L^2(\R^2))}(1+R^s)}{D} \ds\int_0^T \iint (g_n-g)^2(z,\tau) \ud z\ud \tau
 \nonumber \\
\qquad\qquad \ds +e^{C_3T} \frac{2C^{(\eps)}\|\rho\|^2_{L^\infty(0,T;L^2(\R^2))}(\|g_n\|_{L^\infty}+\|g\|_{L^\infty})}{DR^{2-s}}
\ds\int_0^T \iint |z|^2(g_n +g)(z,\tau)\ud z\ud \tau
%
\nonumber\end{array}\]
Pick $\eta>0$. Using the bounds that define
 the set $\mathcal{C}$,  it is possible to select $R(\eta)>0$ such that the last term 
 can be made smaller than $\eta/2$, uniformly with respect to $n$. Then, with this $R$ at  hand, there exists $N(\eta)\in \N$ such that for all  $n\geq N(\eta)$ 
 the first term
in the right hand side is smaller than $\eta/2$ too. 
Hence, 
\[
\int (\rho_n-\rho)^2 (z,t)\ud z
\leq \eta\]
holds for any $n\geq N(\eta)$ and $0\leq t\leq T<\infty$.
It shows that  $\rho_n\to \rho$ in $L^2((0,T)\times \mathbb{R}^2)$. Thus
$\mathcal{T}:\mathcal C\rightarrow \mathcal C$ is continuous for the strong topology of $L^2((0,T)\times \mathbb{R}^2)$.
\qed

\subsection{Preparing for the Schauder theorem: $\mathcal{T}$ is compact}
\label{sec:Comp}

Let $\big(g_n\big)_{n\in \mathbb{N}}$ be a sequence in $\mathcal{C}$, and  set $\rho_n=\mathcal{T}(g_n)$. 
Then, from Lemma \ref{lemma:estrho}, $\rho_n$ is bounded in $L^\infty\left( 0,T; L^2(\mathbb{R}^2)\right)$, and, furthermore,
 $\nabla \rho_n$ is bounded in $L^2\left( (0,T)\times\mathbb{R}^2\right)$.
We also have
\[
\partial_t \rho_n = D \nabla \cdot \nabla \rho_n -\nabla \cdot \left( \vec F^{(\eps)}[g_n]\rho_n\right)
\]
where, by Lemma~\ref{lemma:estF}-i),  
$ \vec F^{(\eps)}[g_n]$ is bounded in $L^\infty$ uniformly with respect to  $n$.
Therefore, $\partial_t \rho_n$ is  bounded in $L^2\left( 0,T; H^{-1}(\mathbb{R}^2)\right)$.
Since the embedding
$H^1(B(0,R))\subset L^2(B(0,R))$ is compact for any $0<R<\infty$, 
we can appeal to the Aubin-Simon lemma, see 
\cite[Cor.~4, Sect.~8]{JS}, to deduce that 
$\big(\rho_n\big)_{n\in\N}$ is relatively compact in $L^2((0,T)\times B(0,R))$
 for any $0<R<\infty$. We need to strengthen this local property to a global  statement.
The moment estimate and the $L^\infty$ estimate in Lemma~\ref{lemma:estrho}-i) and ii) respectively,
 allow us to justify that 
\begin{eqnarray}
\int_0^T\iint_{|z|\geq R} |\rho_n|^2 \ud z\ud t &\leq &
 \frac{\|\rho_n\|_{\infty}}{R^2}\int_0^T\iint |z|^2\rho_n\ud z\ud t \nonumber \leq \ds\frac{C(\eps,T)}{R^2}
\end{eqnarray}
can be made arbitrarily small by choosing $R$ large enough, uniformly with respect to $n\in\N$.
The sequence $\big(\rho_n\big)_{n\in\N}$ thus fulfils the criterion of the Fr\'echet-Weil-Kolmogorov
theorem, see e.~g~\cite[Th. 7.56]{Th} and it is thus  relatively compact in $L^2((0,T)\times\R^2)$.
\qed 

\subsection{Schauder theorem: existence}
Gathering the results of the previous subsections, we can use Schauder's theorem: $\mathcal{C}$ is a closed convex subset of
$L^2((0,T)\times \mathbb{R}^2)$,  $\mathcal{T}$ is a continuous mapping such that $\mathcal{T}(\mathcal{C})\subset\mathcal{C}$ and 
$\mathcal{T}(\mathcal{C})$ is relatively compact in $L^2((0,T)\times \mathbb{R}^2)$. Then $\mathcal{T}$ has a fixed point,
which is a solution of the nonlinear regularized PDE  \eqref{reg_pde}--\eqref{approxF} with initial condition $\rho^{(\eps)}_0$, on any arbitrary time interval $[0,T]$.
The obtained solution lies in $C([0,T];L^2(\R^2))\cap L^2(0,T;H^1(\R^2))$.
\qed

\subsection{Uniqueness}\label{Uniq}

The argument to justify uniqueness relies on the following claim, for which we refer the reader to  \cite[Lemma 7.1.1]{Hale} or \cite[Th.~3.1]{Dix}.

\begin{lemma}[Singular Gr\"onwall Lemma] \label{sing_G} 
Let $A,B\geq 0$, $0\leq \alpha<1$. Let $u(t)$ a locally bounded function such that
\[
u(t) \leq A +B\int_0^t\frac{u(s)}{(t-s)^\alpha}\ud s
\]
then we have 
\[
u(t) \leq A
E_{1-\alpha}\Big(B \Gamma(1-\alpha)t^{1-\alpha}\Big)
\]
with $s\mapsto \Gamma(s)$ the usual $\Gamma-$function and $E_{1-\alpha}$ stands for the Mittag--Leffler function with parameter $\beta=1-\alpha$
\[E_\beta=\ds\sum_{n=0}^\infty \ds\frac{s^n}{\Gamma(n\beta+1)}.\] 
\end{lemma}

Let $\rho_1$ and $\rho_2$ be two solutions of the regularized nonlinear PDE.
Let \begin{equation}\label{heat}
H_t(z)=\ds\frac{1}{4\pi D t}e^{-|z|^2/(4D t)}\end{equation} stand for the two-dimensional  heat kernel with coefficient $D$.
We write
\[\begin{array}{lll}\ds
(\rho_1-\rho_2)(t) &=&\ds H_t\star(\rho_1-\rho_2)(0) - \int_0^t H_{t-s}\star \nabla 
\cdot (\vec F^{(\eps)}[\rho_2]\rho_2-\vec F^{(\eps)}[\rho_1]\rho_1)(s) \ud s\\
&=&\ds  H_t\star(\rho_1-\rho_2)(0) + \int_0^t \nabla H_{t-s}\star  
 (\vec F^{(\eps)}[\rho_2]\rho_2-\vec F^{(\eps)}[\rho_1]\rho_1)(s) \ud s.
\end{array}\]
Initially we have $\rho_2(0)=\rho_1(0)$ and (with $C_0=\frac1\pi \iint |z|e^{-|z|^2}\ud z$) we arrive at 
\begin{eqnarray}
\ds\iint |\rho_1-\rho_2|(z,t) \ud z &\leq &\ds \int_0^t \iint |\nabla H_{t-s}(z-z')|\ |\vec F^{(\eps)}[\rho_2]\rho_2-\vec F^{(\eps)}[\rho_1]\rho_1|(z',s) \ud z'\ud z\ud s \nonumber \\
&\leq &\ds \int_0^t\frac{C_0}{\sqrt{t-s}}\iint |\vec F^{(\eps)}[\rho_2]|\ |(\rho_2-\rho_1))|(z',s) \ud z'\ud s\nonumber \\
&& +\ds \int_0^t\frac{C_0}{\sqrt{t-s}}\int |\vec F^{(\eps)}[\rho_2]-\vec F^{(\eps)}[\rho_1]|\ \rho_1(z',s)|\ud z'\ud s.
\label{eq:unicityL1}
\end{eqnarray}
Lemma \ref{lemma:estF}-i) together with the mass conservation tell us that 
\[
\|\vec F^{(\eps)}[\rho_2]\|_{L^\infty}\leq 
2 M_0 \|\delta^{(\eps)}\|_{L^\infty}  \|\sgne\|_{L^\infty}
.\]
We also have
\begin{eqnarray}
\iint \rho_1(z,t) |\vec F^{(\eps)}[\rho_2]-\vec F^{(\eps)}[\rho_1]|(z,t) \ud z
 &\leq&
 2\|\delta^{(\eps)}\|_{L^\infty}  \|\sgne\|_{L^\infty}
  \iint
  \rho_1(z,t)\ud z \ds\iint |\rho_1-\rho_2|(z',t)\ud z' \nonumber 
  \\&\leq&
  2M_0\|\delta^{(\eps)}\|_{L^\infty}  \|\sgne\|_{L^\infty}\nonumber 
 \ds\iint |\rho_1-\rho_2|(z',t)\ud z'.
\end{eqnarray}
Introducing this into \eqref{eq:unicityL1} yields, for a certain constant  $B>0$:
\begin{eqnarray}
\|\rho_1-\rho_2\|_{L^1}(t) &\leq & B\int_0^t \frac{1}{\sqrt{t-s}} \|\rho_1-\rho_2\|_{L^1}(s)ds
\end{eqnarray}
The singular Gr\"onwall lemma allows us to conclude that $\rho_1=\rho_2$.\qed

\section{Convergence of $\rho^{(\eps)}$}
\label{sec:convergence}

We can now state our main result about the existence of solutions for \eqref{pde}--\eqref{ci}, which is 
expressed as a stability result.

\begin{theorem}\label{main}
Let $\rho^{(\eps)}_0$ be a sequence of non negative functions bounded in $L^1(\R^2)$ and in $L^\infty(\R^2)$.
We suppose that $\iint \rho^{(\eps)}_0\ud y\ud x\leq M_0$ and  $D>2C_2M_0$ (see \eqref{DcritL2}).
Then,  up to a subsequence,  the associated sequence $\big(\rho^{(\eps)} \big)_{\eps>0}$ 
converges strongly in $L^p((0,T)\times \R^2)$ for any $1\leq p<\infty$,
and in $C([0,T];L^p(\R^2)-\textrm{weak})$,
to $\rho$, which is a solution of 
\eqref{pde}--\eqref{def_f} with initial data $\rho_0$, the weak limit of $\rho^{(\eps)}_0$.
\end{theorem}

\subsection{Compactness approach}

We remind the reader that we are assuming  $D>2C_2 M_0$.
Accordingly, from Proposition~\ref{lemma:rho_eps_estimates}, we already know that  $\big(\rho^{(\eps)}\big)_{\eps>0}$
 is bounded in $L^\infty((0,T); L^p(\mathbb{R}^2))$, for any $1\leq p\leq \infty$, and $\nabla \rho^{(\eps)}$ is bounded in $L^2((0,T)\times \mathbb{R}^2)$.
Moreover, the equation
\[
\partial_t \rho^{(\eps)} = \nabla \cdot\left(D\nabla \rho^{(\eps)} -\vec F^{(\eps)}[\rho^{(\eps)}]\rho^{(\eps)} \right) 
\]
tells us that $\partial_t \rho^{(\eps)}$ is the space derivative of the sum of a term bounded in  $L^2((0,T)\times \R^2))$ and the divergence of a term bounded in 
$L^\infty(0,T;L^1(B(0,R)))$ for any $0<R<\infty$.
Indeed, we readily  check that \begin{eqnarray}
 \ds\iint_{B(0,R) } |F^{(\eps)}_x[\rho^{(\eps)}](x,y,t)|\ud y\ud x
  &\leq& \iint_{\sqrt{x^2+y^2}\leq R}\iint \delta^{(\eps)}(y-y')\rho^{(\eps)}(x',y',t) \ud x'\ud y'\ud y \ud x \nonumber \\
&\leq& \ds\int_{-R}^{+R}\ud x \ds\iint \left(\ds\int \delta^{(\eps)}(y-y')\ud y\right) \rho^{(\eps)}(x',y',t) \ud x'\ud y' \\
&\leq & 2RM_0.
\end{eqnarray}
In fact it turns out that $ F^{(\eps)}_x[\rho^{(\eps)}]$, like $ F_x[\rho]$,
is bounded in $L^\infty((0,T)\times \mathbb R;L^1(\mathbb R))$.
Hence,  $\partial_t \rho^{(\eps)}$ is bounded in, say, $L^2(0,T;H^{-1-\delta}(B(0,R)))$ for any $0<R<\infty$ and $\delta >0$.
We can apply the Aubin--Simon lemma \cite{JS}
and we conclude that $\big(\rho^{(\eps)}\big)_{\eps>0}$ is relatively compact in $L^2((0,T)\times B(0,R))$ for any $0<T,R<\infty$.
By using the moments estimate, and  reasoning 
as we did in Section~\ref{sec:Comp}, we show that $\rho^{(\eps)}$ is actually relatively compact in $L^2((0,T)\times \mathbb{R}^2)$. 

Therefore, possibly at the price of extracting a subsequence (still labelled by $\eps$, though) we can assume that 
$$\rho^{(\eps)} \to \rho \textrm{ strongly  in $L^2((0,T)\times \mathbb{R}^2)$}.$$
The convergence can be strengthened in two directions.
First of all, if $1<p=\theta 2+(1-\theta)<2$, the H\"older inequality leads to   $\|\rho^{(\eps)}-\rho\|_{L^p((0,T)\times\R^2)}
\leq (2M_0)^{1-\theta}\| \rho^{(\eps)}-\rho\|^\theta_{L^2((0,T)\times\R^2)}$ and if $2<p<\infty$, we have
$\|\rho^{(\eps)}-\rho\|_{L^p((0,T)\times\R^2)}
\leq (\|\rho^{(\eps)}\|_{L^\infty}+ \|\rho\|_{L^\infty})^{(p-2)/p} \| \rho^{(\eps)}-\rho\|^{2/p}_{L^2((0,T)\times\R^2)}$.
We can also treat the case $p=1$ 
since the $L^2$ estimate and the moment estimate imply that $\big(\rho^{(\eps)}\big)_{\eps>0}$ is weakly compact in $L^1((0,T)\times\R^2)$ 
and
we can assume that it converges a.e., see \cite[Th.~7.60]{Th}.
Finally we get
\begin{equation}\label{cv}
\rho^{(\eps)} \to \rho \textrm{ strongly  in $L^p((0,T)\times \mathbb{R}^2)$ for any $1\leq p<\infty$}.\end{equation}
Second of all, the bound on $\partial_t \rho^{(\eps)}$ can be used to justify, by using the Arzela--Ascoli 
theorem and a diagonal extraction, that 
\[
\lim_{\eps\to 0}
\ds\iint \rho^{(\eps)}(x,y,t)\phi(x,y)\ud y\ud x= \ds\iint \rho(x,y,t)\phi(x,y)\ud y\ud x
\]
holds for any $\phi\in C(\R^2)$, or in $L^{p'}(\R^2)$, uniformly on $[0,T]$.
In particular, the initial data 
passes to the limit and \eqref{ci} makes sense (with $\rho_0$ the weak limit in $L^p(\R^2)$ of the extracted sequence $\big(\rho^{(\eps)}_0\big)_{\eps>0}$).

We are left with the task of passing to the limit in the non--linear term $\vec F^{(\eps)}[\rho^{(\eps)}]\rho^{(\eps)}$. To this end, we  split as follows
$$\vec F^{(\eps)}[\rho^{(\eps)}]-\vec F[\rho]=\vec F^{(\eps)}[\rho^{(\eps)}-\rho]+(\vec F^{(\eps)}[\rho]-\vec F[\rho]).$$
The first term tends to 0
as a consequence of \eqref{cv} combined with the following claim.

\begin{lemma}
The operator $ F^{(\eps)}_x $ (resp. $F^{(\eps)}_y $) is, uniformly with respect to $\eps$, continuous from $L^1(\R^2)$ to $L^\infty(\R_x;L^1(\R_y))$
(resp. $L^\infty(\R_y;L^1(\R_x))$).
\end{lemma}

\noindent
{\bf Proof.}
For any $\phi\in L^1(\R^2)$, we have
\begin{eqnarray}
\int |F_x^{(\eps)}[\phi](x,y)| \ud y &\leq & \int \left( \iint \delta^{(\eps)}(y-y') |\phi(x',y')| \ud x'\ud y'\right) \ud y \nonumber \\
 &\leq & \iint   \left( \int \delta^{(\eps)}(y-y')\ud y\right) 
 |\phi(x',y')| \ud x'\ud y'
= \|\phi\|_{L^1}.\nonumber
\end{eqnarray}
\qed

It remains to investigate, for $\phi\in L^1(\R^2)$, 
the behavior of 
\[\begin{array}{lll}
\big|F^{(\eps)}_x[\phi]-F_x[\phi]\big|(x,y)
&=&
\big|\ds\iint \delta^{(\eps)}(y-y')\sgne(x-x') \phi(x',y')\ud y'\ud x'\\&&
-\ds\int  \left(\ds\int \delta^{(\eps)}(y-y')\ud y'\right) \sgn(x-x') \phi(x',y)\ud x'\big|
\\
&\leq&
\ds\iint \delta^{(\eps)}(y-y')
\big|\sgne(x-x') \phi(x',y')-\sgn(x-x') \phi(x',y)
\big|\ud y'\ud x'.
\end{array}\]
We integrate 
with respect to $y$ and, bearing in mind that $\delta^{(\eps)}(y)=\frac1\eps \delta(y/\eps)$ with $\delta$ the normalized Gaussian, we use the change of variable $y-y'=\eps \xi$; it yields
\[\begin{array}{lll}
\ds\int \big|F^{(\eps)}_x[\phi]-F_x[\phi]\big|(x,y)\ud y
&\leq &
\ds\iiint \delta(\xi ) \big| \sgne(x-x') \phi(x',y-\eps \xi )- \sgn(x-x') \phi(x',y)
\big|\ud \xi \ud x'\ud y
\\
&\leq &
\ds\iiint \delta(\xi ) \big|  \phi(x',y-\eps \xi )-  \phi(x',y)
\big|\ud \xi \ud x'\ud y
\\&&
+
\ds\iiint \delta(\xi ) \phi(x',y) \big| \sgne(x-x')- \sgn(x-x') 
\big|\ud \xi \ud x'\ud y.
\end{array}\]
On the right hand side, the first integral recasts as 
\[
\ds\int \delta (\xi)\left(
\ds\iint \big|  \phi(x',y-\eps \xi )-  \phi(x',y)\big|\ud x'\ud y,
\right)\ud \xi
\]
which tends to 0 as $\eps\to 0$ by combining the Lebesgue dominated convergence theorem with the continuity of translation in $L^1$, \cite[Cor. 4.14]{Th}.
The second integral reads
\[
\ds\int \delta (\xi)\ud \xi\times \ds\iint
 \phi(x',y) \big| \sgne(x-x')- \sgn(x-x') 
\big|\ud x'\ud y
\]
The function $(x,x')\mapsto |\sgne(x-x')-\sgn(x-x')|$ tends  to $0$ pointwise and it is dominated by  $2$.
Since  $\rho\in L^1(\mathbb R^2)$, 
a direct application of the Lebesgue dominated convergence theorem tells us that this quantity tends to 0 as $\eps\to 0$, for any given $x\in \R$.
Similar reasoning obviously apply 
to the  second component of $\vec F$.
Finally, for any test function $\varphi\in C^\infty_c(\mathbb R^2)$,
we obtain 
$$\ds\lim_{\eps\to 0}
\ds\iint \varphi \left(F^{(\eps)}[\rho^{(\eps)}]\rho^{(\eps)}-F[\rho]\rho\right)\ud y\ud x=0.$$ 
Therefore  $\rho$ satisfies, in a weak sense, the limit  equation
\eqref{pde}--\eqref{def_f}. This completes the proof of Theorem~\ref{main}.\qed

\subsection{Symmetric solutions}

Throughout this Section, we work with data that satisfy the following symmetry condition
\begin{equation}\label{sym}
\rho_0(-x,y)=\rho_0(x,y)=\rho_0(x,-y).
\end{equation}
It will be used to derive further estimates and a stronger convergence result of the regularized solution $\rho^{(\eps)}$ towards the solutions of 
\eqref{pde}--\eqref{ci}.
Using the uniqueness property of the solution of the regularized equation \eqref{reg_pde}--\eqref{approxF}, we deduce that the  symmetry property is preserved by the solutions of \eqref{pde}. 
Accordingly, we get
\[
F_x[\rho] (0,y,t) =-\int \sgn(x') \rho(x',y,t)\ud x' =0,\qquad F_y[\rho](x,0,t) = 0.
\]
However, we know that $\partial_x F_x[\rho]<0$ and $\partial_y F_y[\rho]<0$. 
Thus, $x\mapsto F_x[\rho](x,y,t)$ is non increasing and it vanishes for $x=0$, so that it has the sign of $(-x)$.
We deduce that $$(x,y)\cdot \vec F[\rho](x,y,t)
=xF_x[\rho] (x,y,t)+yF_y[\rho] (x,y,t) \leq 0.$$
A similar property hold with the solutions $\rho^{(\eps)}$ of the regularized problem and the force operator $\vec F^{(\eps)}[\rho^{(\eps)}]$.
This will be used to obtain a strengthened control on the behavior 
of the solutions for large $x,y$'s: exponential moments and
weighted estimates on the gradients.
These estimates will be combined with the interpretation of \eqref{pde} as a perturbation of the heat equation. Namely, still with $H_t$ the heat kernel
\eqref{heat}, we shall make use of the Duhamel formula
\begin{equation}\label{Duhamel}
\rho(x,y,t)=H_t\star\rho_0(x,y)- \ds\int_0^t H_{t-s}\star \nabla\cdot \big(\vec F[\rho]\rho(s,\cdot)\big)(x,y)\ud s,\end{equation}
and the analogous formula with $\rho^{(\eps)}$.

\subsubsection{Strengthened estimates for symmetric solutions}
At first, the symmetry property allows us to control exponential moments. 
\begin{lemma}[Exponential moments]
Assume that $\rho_0$ satisfies \eqref{sym} and
$$
\ds\iint e^{\lambda\sqrt{1+x^2+y^2}}\rho_0(x,y)\ud y\ud x=\mathcal E_0(\lambda)<+\infty$$
for some $\lambda>0$. Then, the solutions of \eqref{pde}--\eqref{ci} satisfy
\[
\iint e^{\lambda\sqrt{1+x^2+y^2}}\rho(x,y,t) \ud y\ud x\leq \mathcal E_0(\lambda)  e^{D(\lambda^2+2\lambda)t}.
\]
The same estimate holds replacing $\rho$ by $\rho^{(\eps)}$.
\end{lemma}

\noindent
{\bf Proof.} 
By using integration by parts, we get
\[\begin{array}{lll}
\ds
\frac{\ud}{\ud t} \iint e^{\lambda\sqrt{1+x^2+y^2}}\rho(x,y,t) \ud y\ud x & \leq& 
\ds D(\lambda^2+2\lambda)\iint e^{\lambda\sqrt{1+x^2+y^2}}\rho(x,y,t) \ud y\ud x\\
 &&\ds+\iint \lambda e^{\lambda\sqrt{1+x^2+y^2}}\ \frac{(x,y)\cdot{ \vec F}[\rho]}{\sqrt{1+x^2+y^2}} \rho(x,y,t)\ud y\ud x.
\end{array}\]
As consequence of the symmetry assumption, the last term contributes negatively. We end the proof by integrating with respect to time. \qed

Using $L^q$ and moments estimates, we can readily obtain a weighted $L^2$ bound; for instance, we have 
\begin{eqnarray}
\iint e^{\lambda\sqrt{1+x^2+y^2}}\rho^2(x,y,t)\ud y\ud x &\leq& \left(\iint e^{2\lambda\sqrt{1+x^2+y^2}}\rho(x,y,t)\ud y\ud x\right)^{1/2}\left(\iint \rho^3\ud y\ud x\right)^{1/2} \nonumber 
\end{eqnarray} 
and a similar estimate holds for $\rho^{(\eps)}$. According to Proposition~\ref{apriori}-i) \& iii)  and \ref{lemma:rho_eps_estimates}-i) \& iii), it becomes a relevant estimate for $D$ large enough:
when \eqref{DcritL2} holds we have bounds in $L^1(\R^2)\cap L^\infty(\R^2) $, thus on $L^3(\R^2)$.
We finally arrive at 
\begin{equation}\label{mtexpcarre}
\ds\iint e^{\lambda\sqrt{1+x^2+y^2}}\rho^2(x,y,t)\ud y\ud x \leq
C e^{2D\lambda(1+\lambda)t}
\end{equation}
where the constant $C$ depends on $D$, $\mathcal E_0(2\lambda)$,  $\|\rho_0\|_{L1} $ and $\|\rho_0\|_{L^\infty}$.
Again, the same (uniform) estimate is fulfilled by $\rho^{(\eps)}$.
\qed
%
%
%

We need now to specify the class of initial data to which the analysis applies.
Addtionally to the symmetry assumption, we suppose that $\rho_0^{(\eps)}$, which is a regularization of $\rho_0$ in \eqref{ci}, is such that
\begin{equation}\label{hypci}
\begin{array}{l}
\textrm{there exists $p_0,p_2\geq 0$ such that for any $\lambda\geq 0$, we have 
}
\\
\hspace*{4cm}
\mathcal E_0(\lambda)=\ds\sup_{\eps>0}\left(\ds\iint e^{\lambda\sqrt{1+x^2+y^2}}\rho^{(\eps)}_0(x,y)\ud y\ud x\right)\leq e^{p_0+p_2\lambda^2}.
\end{array}
\end{equation}
Such an assumption clearly holds for uniformly compactly supported data, as well as for Gaussian--like data.
Finally, for our purpose, we will need another estimate for the weighted $L^2$ norm of the gradient, which applies for the data verifying  \eqref{hypci}.

 \begin{lemma}[Weighted $L^2$ estimates]\label{west}
 Let 
 $\big(\rho^{(\varepsilon)}_0\big)_{\eps>0}$ be a sequence of non negative functions  bounded in $L^1\cap L^\infty(\R^2)$.
Assume that $\rho^{(\eps)}_0$ satisfies \eqref{sym} and 
\[\ds\sup_{\eps>0}\left\{\int |z|^4\rho^{(\eps)}_0 \ud z + \iint (1+|z|^2)|\rho^{(\eps)}_0|^2 \ud z\right\}<\infty.\]
If $D$ satisfies \eqref{DcritL2}, there exists constants  $B_0, B_1, B_2>0$, which do not depend on $\eps$ 
nor on $t$, such that
\[\begin{array}{l}\ds
\iint (1+|z|^2)(\rho^{(\eps)})^2 \ud z \leq B_0+B_1t
+B_2 t^2,
\\
\ds\int_0^t\left(\iint (1+|z|^2)|\nabla \rho^{(\eps)}(z,s)|^2\right) \ud s\leq B_0+B_1t
+B_2 t^2.
\end{array}\]
\end{lemma}

\noindent
{\bf Proof.} Let us compute (still with the shorthand notation $z=(x,y)$), by using several integrations by parts,
\begin{eqnarray}
\ds\frac12\frac{\ud}{\ud t} \iint(1+|z|^2)|\rho^{(\eps)}|^2(z,t)\ud z &=&-D\ds\iint (1+|z|^2)|\nabla\rho^{(\eps)}|^2\ud z +2D \iint |\rho^{(\eps)}|^2\ud z 
\nonumber \\
&&\ds
+\iint |\rho^{(\eps)}|^2 z\cdot F^{(\eps)}[ \rho^{(\eps)}] \ud z 
 \nonumber\\
 &&- \ds\frac 12 \iint(1+|z|^2)
| \rho^{(\eps)}|^2 \ \nabla\cdot( \vec F^{(\eps)} \rho^{(\eps)})\ud z. \nonumber 
\end{eqnarray}
The symmetry assumption implies  $z\cdot F^{(\eps)}[ \rho^{(\eps)}]\leq 0$, which allows us to get rid of the third term in the right side.
We remind the reader that $\nabla\cdot( \vec F^{(\eps)} \rho^{(\eps)})=-4T^{(\eps)}( \rho^{(\eps)})$ is proportional to  the convolution with
an approximation of  the 2D Dirac measure.
Hence, we get
\[
\begin{array}{l}
\ds\frac12\frac{\ud}{\ud t} \iint(1+|z|^2)|\rho^{(\eps)}|^2(z,t)\ud z +D\ds\iint (1+|z|^2)|\nabla\rho^{(\eps)}|^2\ud z
\\
\qquad
\leq
2D \ds\iint |\rho^{(\eps)}|^2\ud z + 2\ds \iint(1+|z|^2)
| \rho^{(\eps)}|^2\  T^{(\eps)}( \rho^{(\eps)})\ud z
\\
\qquad \leq \ds
 2D\iint|\rho^{(\eps)}|^2\ud z
  +2\left(\iint (1+|z|^2)^2\rho^{(\eps)} \ud z \right)^{1/2}\left(\iint \big|T^{(\eps)}(\rho^{(\eps)})\big|^2\ \big|\rho^{(\eps)}\big|^3\ud z\right)^{1/2} \nonumber \\
\qquad\ds\leq  2D\iint|\rho^{(\eps)}|^2\ud z +2\left(2\iint (1+|z|^4)\rho^{(\eps)} \ud z \right)^{1/2}
\left(\iint |\rho^{(\eps)}|^5 \ud z\right)^{1/2} .\nonumber 
\end{array}\]
For the last term, we have used 
H\"older's inequality
as in the proof of Proposition~\ref{lemma:rho_eps_estimates}.
We already know that the $L^2$ and $L^5 $ norms of $\rho^{(\eps)}$ are uniformly bounded, by virtue of Proposition~\ref{lemma:rho_eps_estimates}.
It remains to discuss the forth order moment. To this end we go back to \eqref{mtk}: $t\mapsto m_2(t) $ has a linear growth, hence $t\mapsto m_4(t)$ have a quadratic growth  
with respect to the time variable.
We conclude that both 
\[
\iint(1+|z|^2)|\rho^{(\eps)}|^2(z,t)\ud z \qquad \textrm{and} \qquad \ds\int_0^t\ds\iint (1+|z|^2)|\nabla\rho^{(\eps)}(z,s)|^2\ud z\ud s
\]
has at most a quadratic growth, with coefficients independent of $\eps$.
\qed

\subsubsection{Cauchy property for $\rho^{(\eps)}$}
\label{sec:cauchy}

This Section is concerned with the following statement, 
which strengthens Theorem~\ref{main} for symmetric solutions.

\begin{theorem}\label{main2}
 Let 
 $\big(\rho^{(\varepsilon)}_0\big)_{\eps>0}$ be a sequence of non negative functions  bounded in $L^1\cap L^\infty(\R^2)$,
 which 
 satisfies \eqref{sym} and \eqref{hypci}
and which converges in $L^1(\R^2)$ to some $\rho_0$.
Then the associated sequence $\big(\rho^{(\varepsilon)}\big)_{\eps>0}$ of solutions of \eqref{reg_pde}--\eqref{approxF}
is a Cauchy sequence in $C([0,T];L^1(\mathbb{R}^2))$.
\end{theorem}
 
 \begin{coro}\label{main2b}
Let $\rho_0\in L^1\cap L^\infty(\R^2)$ verify \eqref{sym} and \eqref{hypci}.
Then the sequence $\big(\rho^{(\eps)}\big)_{\eps>0}$ of solutions of \eqref{reg_pde}--\eqref{approxF} with the same initia data converges $C([0,T];L^1(\mathbb R^2))$
 to $\rho$, the unique   symmetric solution of \eqref{pde}--\eqref{ci}.
\end{coro}

 We make use of \eqref{Duhamel}, which leads to 
\begin{equation}
\begin{array}{lll}
\ds\iint |\rho^{(\eps)}-\rho^{(\eps')}|(z,t)\ud z &\leq &\ds \int H_t\star  |\rho^{(\eps)}_0-\rho^{(\eps')}_0|(z)\ud z 
 \\&& +\ds\int_0^t  \iint  \left|\nabla H_{t-s} \star\left[ \left(\vec F^{(\eps)}[\rho^{(\eps)}]\rho^{(\eps)} 
-\vec F^{(\eps')}[\rho^{(\eps')}]\rho^{(\eps')} \right) (s,\cdot)\right]\right|(z)\ud z \ud s. \label{int112}
\end{array}
\end{equation}

We dominate the right hand side by the sum of  the following four terms
\begin{eqnarray}
A_{\eps,\eps'}(t)&= &\iint_{\mathbb{R}^2} H_t \star |\rho^{(\eps)}_0-\rho^{(\eps')}_0|(z)\ud z,\nonumber
\\
B_{\eps,\eps'}(t)&= & \int_0^t \iiiint |\nabla H_{t-s}(z-z')|\
 |\vec F^{(\eps)}[\rho^{(\eps)}](z',s)| \ 
 |\rho^{(\eps)}-\rho^{(\eps')}|(z',s)\ud z' \ud z\ud s, \nonumber \\
C_{\eps,\eps'}(t)&= &\int_0^t \iiiint |\nabla H_{t-s}(z-z')|\ \rho^{(\eps')}(z',s)
\ |\vec F^{(\eps)}[\rho^{(\eps)} -\rho^{(\eps')}](z',s)|
\ud z'\ud z
\ud s, \nonumber \\
D_{\eps,\eps'}(t)&= &\int_0^t \iint \left| \iint  \rho^{(\eps')}(z',s)\  \nabla H_{t-s}(z-z')\cdot  
\Big(\vec F^{(\eps)}[\rho^{(\eps')}]- \vec F^{(\eps')}[\rho^{(\eps')}](z',s)\Big)\ud z'\right| \ud z\ud s.
\nonumber
\end{eqnarray}
Since $\rho^{(\eps)}_0\to \rho_0$ in $L^1(\R^2)$, it is clear that
\begin{equation}\label{Acv}
\ds\lim_{\eps,\eps'\to 0}\left(\ds\sup_{t\geq 0} A_{\eps,\eps'}(t)\right)=0
\end{equation}
uniformly on any time interval $[0,T]$.
Next, we are going to justify the following claim.

\begin{lemma}\label{l2} 
Let 
$\alpha=\frac{1}{\sqrt2}
$. Set  $$\varphi(\lambda)=2D\lambda(1+\alpha)\big(1+\lambda(1+\alpha)\big).$$
Then there exists  constant $\beta_1,\beta_2>0$ such that, for any $R>0$ we have 
\begin{equation}\label{Bcv}
 B_{\eps,\eps'}(t)
\leq    \beta_1 R \|\rho^{(\eps)}\|_{L^\infty}\ds\int_0^t \ds\frac{1}{\sqrt{t-s}} \|(\rho^{(\eps)}-\rho^{(\eps')})(s,\cdot)\|_{L^1}\ud s
  +\beta_2\ds\frac{\sqrt{t}}{\lambda} \ e^{-\alpha \lambda R}e^{\varphi(\lambda)t}.
\end{equation}
The constant $\beta_1$ does not depend on the data, while $\beta_2$ depends on  $\mathcal E_0(2\lambda(1+\alpha))$.
\end{lemma}

\noindent
{\bf Proof.} 
In Section \ref{Uniq}, we already used the basic estimate
\begin{equation}\label{estheatk}
\iint |\nabla H_{t-s}(z-z')| \ud z' \leq \frac{C_0}{\sqrt{t-s}}
\end{equation}
for a certain constant $C_0$.
We have
\[\begin{array}{l}\ds
\iint  |F^{(\eps)}[\rho^{(\eps)}](z')| |\rho^{(\eps)}-\rho^{(\eps')}|(z')\ud z' 
\\
\qquad\ds
\leq  \iiiint \rho^{(\eps)}(x_1,y_1)\delta^{(\eps)}(y'-y_1)|\rho^{(\eps)}-\rho^{(\eps')}|(x',y')\ud x'\ud y'\ud x_1\ud y_1 
\\
\qquad\leq 
\ds
\ds\iiiint_{x_1^2+y_1^2\leq R^2}... \ud x'\ud y'\ud x_1\ud y_1+ \ds\iiiint_{x_1^2+y_1^2\geq R^2}... \ud x'\ud y'\ud x_1\ud y_1
.\end{array}
\]
We dominate the first integral as follows
\[\begin{array}{l}
\ds\iiiint_{x_1^2+y_1^2\leq R^2}... \ud x'\ud y'\ud x_1\ud y_1\\
\qquad\qquad\leq 
\|\rho^{(\eps)}\|_{L^\infty} \ds\int_{x_1\leq R} \left(\ds\int \delta^{(\eps)}(y'-y_1)\ud y_1\right)|\rho^{(\eps)}-\rho^{(\eps')}|(x',y') \ud y'\ud x'
\\ \qquad\qquad\leq 
2R\|\rho^{(\eps)}\|_{L^\infty}\|\rho^{(\eps)}-\rho^{(\eps')}\|_{L^1}.\end{array}\]
Next, we have
\[\begin{array}{l}
\ds\iiiint_{x_1^2+y_1^2\geq R^2}... \ud x'\ud y'\ud x_1\ud y_1
\\
\qquad\qquad
\leq
\ds\iiiint_{x_1^2+y_1^2\geq R^2}
\rho^{(\eps)}(x_1,y_1)\rho^{(\eps)}(x',y') \delta^{(\eps)}(y'-y_1)
\ud x'\ud y'\ud x_1\ud y_1
\\
\qquad\qquad\qquad
+\ds\iiiint_{x_1^2+y_1^2\geq R^2}
\rho^{(\eps)}(x_1,y_1)\rho^{(\eps')}(x',y') \delta^{(\eps)}(y'-y_1)
\ud x'\ud y'\ud x_1\ud y_1
\end{array}\]
where the two terms can be treated with the same approach.
We make the exponential moment appear and we use the Cauchy-Schwarz inequality
to obtain,
for instance,
\begin{equation}\label{toto}
\begin{array}{l}
\ds\iiiint_{x_1^2+y_1^2> R^2} \rho^{(\eps)}(x_1,y_1) \delta^{(\eps)}(y-y_1)\rho^{(\eps)}(x,y)\ud x_1\ud y_1\ud x\ud y
\\
\qquad\leq\ds
 \frac12\ds\iiiint_{x_1^2+y_1^2> R^2}  e^{\lambda\sqrt{1+x_1^2+ y_1^2}}e^{-\lambda\sqrt{1+x^2+y^2}}|\rho^{(\eps)}|^2(x_1,y_1)
 \delta^{(\eps)}(y-y_1)
\ud x_1\ud y_1\ud x\ud y \\
\qquad\qquad + \ds\frac12 \ds\iiiint_{x_1^2+y_1^2> R^2} 
e^{-\lambda\sqrt{1+x_1^2+y_1^2}}e^{\lambda\sqrt{1+x^2+y^2}}|\rho^{(\eps)}|^2(x,y)\delta^{(\eps)}(y-y_1)
\ud x_1\ud y_1\ud x\ud y .
\end{array}\end{equation}
The elementary inequality  $$\alpha(|x|+|y|)\leq \sqrt{1+x^2+y^2}$$ allows us to estimate
\[
\int e^{-\lambda\sqrt{1+x^2+y^2}}\ud x \leq e^{-\alpha\lambda |y|}\frac{2}{\alpha \lambda}
.\]
Hence the first integral in the right hand side of \eqref{toto} is dominated by
\[\begin{array}{l}
\ds\iint_{x_1^2+y_1^2> R^2} 
\left(\ds\int  \delta^{(\eps)}(y-y_1)
 \left(\ds\int e^{-\lambda\sqrt{1+x^2+y^2}} \ud x \right)\ud y
\right)
e^{\lambda\sqrt{1+x_1^2+ y_1^2}} \ |\rho^{(\eps)}|^2(x_1,y_1)
\ud x_1\ud y_1
\\
\qquad\leq 
\ds\frac{2}{\alpha\lambda}
\ds\iint_{x_1^2+y_1^2> R^2} 
\left(\ds\int  \delta^{(\eps)}(y-y_1) e^{-\alpha\lambda |y|}
\ud y
\right)
e^{\lambda\sqrt{1+x_1^2+ y_1^2}} \ |\rho^{(\eps)}|^2(x_1,y_1)
\ud x_1\ud y_1\
\\
\qquad\leq 
\ds\frac{2}{\alpha\lambda}
\ds\iint_{x_1^2+y_1^2> R^2}e^{\lambda\sqrt{1+x_1^2+ y_1^2} }\ |\rho^{(\eps)}|^2(x_1,y_1)\ud x_1\ud y_1
\\
\qquad\leq 
\ds\frac{2}{\alpha\lambda}
\ds\iint_{x_1^2+y_1^2> R^2}e^{-\lambda\alpha\sqrt{1+x_1^2+ y_1^2} }
e^{\lambda(1+\alpha)\sqrt{1+x_1^2+ y_1^2} }\ |\rho^{(\eps)}|^2(x_1,y_1)\ud x_1\ud y_1
\\
\qquad\leq 
\ds\frac{2}{\alpha\lambda} \ e^{-\lambda\alpha R}
\ds\iint
e^{\lambda(1+\alpha)\sqrt{1+x_1^2+ y_1^2} }\ |\rho^{(\eps)}|^2(x_1,y_1)\ud x_1\ud y_1
\\
\qquad\leq C \ 
\ds\frac{2}{\alpha\lambda} \ e^{-\lambda\alpha R}\ 
e^{2D\lambda(1+\alpha)(1+(1+\alpha)\lambda)t}
\end{array}\]
where we have used 
\eqref{mtexpcarre} and the constant $C>0$ depends on $\mathcal E_0(2\lambda(1+\alpha))$.
Next, we observe that
\begin{eqnarray}
\ds\iint \mathbf 1_{x_1^2+y_1^2> R^2} e^{-\lambda\sqrt{1+x_1^2+y_1^2}}\ud x_1
 \leq e^{-\alpha\lambda |y_1|}\int \mathbf 1_{x_1^2+y_1^2> R^2} e^{-\alpha \lambda |x_1|} \ud x_1\nonumber 
\leq  \frac{2}{\alpha \lambda} e^{-\alpha \lambda( |y_1|+ g(y_1))} \nonumber 
\end{eqnarray}
where $$g(y)=\mathbf 1_{|y|\leq R}\sqrt{R^2-y^2}.$$
As a matter of fact, for any $y\in \R$, 
we have $|y|+g(y)\geq R$, so that 
\[
\int \mathbf 1_{x_1^2+y_1^2> R^2} e^{-\lambda\sqrt{1+x_1^2+y_1^2}}\ud x_1 \leq  \frac{2}{\alpha \lambda} e^{-\alpha \lambda R}
.\]
The second integral of the right hand side in \eqref{toto}, 
is thus dominated by 
\[\begin{array}{l}
\ds\iint
\left(\ds\int \left(\ds\int \mathbf 1_{x_1^2+y_1^2> R^2} e^{-\lambda\sqrt{1+x_1^2+y_1^2}}\ud x_1\right)
\delta^{(\eps)}(y-y_1)
\ud y_1
\right)
e^{\lambda\sqrt{1+x^2+y^2}} |\rho^{(\eps)}|^2(x,y)\ud x\ud y
\\
\qquad
\leq \ds
 \frac{2}{\alpha \lambda} e^{-\alpha \lambda R}
 \ds\iint
\left(\ds\int 
\delta^{(\eps)}(y-y_1)
\ud y_1
\right)
e^{\lambda\sqrt{1+x^2+y^2}} |\rho^{(\eps)}|^2(x,y)\ud x\ud y
\\
\qquad\leq 
\ds C'\ \frac{2}{\alpha \lambda} e^{-\alpha \lambda R}\ e^{2D\lambda(1+ \lambda)t}
\end{array}\]
where we have used \eqref{mtexpcarre} again and $C'$ here depends on $\mathcal E_0(2\lambda)$.
We finally  conclude (note that $\lambda(1+\alpha)>\lambda$) that 
\[\begin{array}{l}
\ds\iint  |F^{(\eps)}[\rho^{(\eps)}](z',s)| |\rho^{(\eps)}-\rho^{(\eps')}|(z',s)\ud z' 
\\
\\\qquad
\leq  2 R \|\rho^{(\eps)}\|_{L^\infty} \|(\rho^{(\eps)}-\rho^{(\eps')})(s)\|_{L^1} +
\ds\frac{C}{\alpha \lambda}\ e^{-\alpha \lambda R}e^{\varphi(\lambda)t}
\nonumber
\end{array}\]
with $C$ depending on $\mathcal E_0(2\lambda(1+\alpha))\geq \mathcal E_0(2\lambda)$.
We combine this inequality  to \eqref{estheatk} to obtain the final estimate on $B_{\eps,\eps'}$.
 \qed

\begin{lemma} \label{l3} There exists  constant $\gamma_1,\gamma_2>0$ such that, for any $R>0$ we have 
\begin{equation}\label{Ccv}
 C_{\eps,\eps'}(t)
\leq   
   \gamma_1 R \|\rho^{(\eps)}\|_{L^\infty}\ds\int_0^t \ds\frac{1}{\sqrt{t-s}} \|(\rho^{(\eps)}-\rho^{(\eps')})(s,\cdot)\|_{L^1}\ud s
  +\gamma_2\ds\frac{\sqrt{t}}{\lambda} \ e^{-\alpha \lambda R}e^{\varphi(\lambda)t}.
\end{equation}
The constant $\gamma_1$ does not depend on the data, while $\gamma_2$ depends on  $\mathcal E_0(2\lambda(1+\alpha))$.
\end{lemma}

\noindent
{\bf Proof.}
The same reasoning  applies for $C_{\eps,\eps'}$. Indeed, we can first integrate  $\nabla H_{t-s}(z-z')$ over $z$,
which leads to the analog of \eqref{estheatk}.
Estimating $\vec{F}^{(\eps)}$, we are left with
\[
C_{\eps,\eps'}(t) \leq\ds \int_0^t \frac{C_0}{\sqrt{t-s}} \iint \rho^{(\eps')}(x',y',s) |\rho^{(\eps)}-\rho^{(\eps')}|(x_1,y_1,s))\delta^{(\eps)}(y'-y_1)\ud x_1\ud y_1\ud x'\ud y'\ud s
\]
which is exactly the same expression that appeared in the analysis of  $B_{\eps,\eps'}$.
\qed

\noindent
We turn to the analysis of $D_{\eps,\eps'}$.

\begin{lemma}\label{l4} Let $0<T<\infty$.
Then $
 D_{\eps,\eps'}(t)$ converges to 0, uniformly over $[0,T]$ as $\eps,\eps'$ tend to $0$
\end{lemma}

\noindent
{\bf Proof.}
We evaluate $
 D_{\eps,\eps'}$  through the following splitting
\begin{eqnarray}
D_{\eps,\eps'}(t)&\leq &
D_{x,1}(t)+D_{x,2}(t)+D_{y,1}(t)+D_{y,2}(t)\nonumber 
\end{eqnarray}
with 
\begin{eqnarray}
D_{x,1}(t) &=& \int_0^t \iint \left| \iiiint \partial_x H_{t-s}(x-x',y-y')
\rho^{(\eps')}(x',y',s)\rho^{(\eps')}(x",y",s)
\right.\nonumber \\
&&\qquad\qquad \times \left. 
\sgn^{(\eps)}(x'-x")\big(\delta^{(\eps)}(y'-y")-\delta^{(\eps')}(y'-y")\big)
\ud x'\ud y'\ud x"\ud y"\right| \ud x\ud y \nonumber \\
D_{x,2}(t) &=& \int_0^t \iint \left| \iiiint \partial_x H_{t-s}(x-x',y-y) 
\rho^{(\eps')}(x',y',s)\rho^{(\eps')}(x",y",s) 
\right.\nonumber \\
&&\qquad\qquad \times \left.
\delta^{(\eps')}(y'-y")\big(\sgn^{(\eps)}(x'-x")-\sgn^{(\eps')}(x'-x")\big)
\ud x'\ud y'\ud x"\ud y"\right| \ud x\ud y\nonumber \\
D_{y,1}(t) &=& \int_0^t \iint \left| \iiiint \partial_x H_{t-s}(x-x',y-y')
\rho^{(\eps')}(x',y',s)\rho^{(\eps')}(x",y",s)
\right.\nonumber \\
&&\qquad\qquad \times \left. 
\sgn^{(\eps)}(y'-y")\big(\delta^{(\eps)}(x'-x")-\delta^{(\eps')}(x'-x")\big)
\ud x'\ud y'\ud x"\ud y"\right| \ud x\ud y \nonumber \\
D_{y,2}(t) &=& \int_0^t \iint \left| \iiiint \partial_y H_{t-s}(x-x',y-y) 
\rho^{(\eps')}(x',y',s)\rho^{(\eps')}(x",y",s) 
\right.\nonumber \\
&&\qquad\qquad \times \left.
\delta^{(\eps')}(x'-x")\big(\sgn^{(\eps)}(y'-y")-\sgn^{(\eps')}(y'-y")\big)
\ud x'\ud y'\ud x"\ud y"\right| \ud x\ud y.\nonumber 
\end{eqnarray}
In order to study $D_{x,1}$,
we make use of the following quantity
\[\ds
\iint \partial_x H_{t-s}(x-x',y-y')
\rho^{(\eps')}(x',y',s) \mathcal I_{\eps,\eps'}(x',y',s)
\ud x'\ud y'\]
with
\[\mathcal I_{\eps,\eps'}(x',y',s)=
\ds\iint
\rho^{(\eps')}(x",y",s)
\sgn^{(\eps)}(x'-x")\big(\delta^{(\eps)}(y'-y")-\delta^{(\eps')}(y'-y")\big)
\ud x"\ud y".\]
Since $\delta^{(\eps)}(u)=\frac12\frac{\ud}{\ud u}\sgn^{(\eps)}(u)$, the latter can be rewritten
by integrating by parts
\[\mathcal I_{\eps,\eps'}(x',y',s)=
\ds\frac12\ds\iint
\partial_y\rho^{(\eps')}(x",y",s)
\sgn^{(\eps)}(x'-x")\big(\sgn^{(\eps)}(y'-y")-\sgn^{(\eps')}(y'-y")\big)
\ud x"\ud y".\]
The 
Cauchy-Schwarz inequality yields
\[\begin{array}{lll}
|\mathcal I_{\eps,\eps'}(x',y',s)|&\leq&
\ds\frac12\left(\ds\iint (1+|x"|^2)
\big|\partial_y\rho^{(\eps')}(x",y",s)\big|^2\ud x"\ud y"\right)^{1/2}
\\
&&\qquad\times
\left(\ds\int \ds\frac{\ud x"}{1+|x"|^2}\ds\int
\big|\sgn^{(\eps)}(y'-y")-\sgn^{(\eps')}(y'-y")\big|^2\ud y"\right)^{1/2}
\\
&\leq & 
\ds\frac{\sqrt\pi}{2}\left(\ds\iint (1+|z"|^2)
|\nabla\rho^{(\eps')}(z",s)|^2\ud z"\right)^{1/2}\sqrt{\Delta_{\eps,\eps'}(y')},
\end{array}\]
where 
\[\begin{array}{lll}\Delta_{\eps,\eps'}(y')&=&\ds\int
\big|\sgn^{(\eps)}(y'-y")-\sgn^{(\eps')}(y'-y")\big|^2\ud y"
\\
&=&
\ds\frac2\pi
\ds\int\Big| \ds\int_0^{y'-y"} e^{-\frac{v^2}{2\eps^2}}\ds\frac{\ud v}{\eps}-  \ds\int_0^{y'-y"} e^{-\frac{v^2}{2|\eps'|^2}}\ds\frac{\ud v}{\eps'} \Big|^2\ud y"
\\
&=&
\ds\frac2\pi
\ds\int\Big| \ds\int_{u/\eps'}^{u/\eps} e^{-v^2/2}\ud v
 \Big|^2\ud u
.
\end{array}\]
In particular this quantity does not depend on $y'$.
Clearly, for any fixed $u\in\R$, we have
\[\ds\lim_{\eps,\eps'\to 0}\left(\ds\int_{u/\eps'}^{u/\eps} e^{-v^2/2}\ud v\right)=0.\]
Furthermore, for $0<\eps,\eps'\ll 1$, it can be dominated as follows
\[\left|\ds\int_{u/\eps'}^{u/\eps} e^{-v^2/2}\ud v\right|
=\left|\ds\int_{u/\eps'}^{u/\eps} e^{-v^2/4}\ e^{-v^2/4}\ud v\right|
\leq e^{-u^2/2}\ds\int e^{-v^2/4}\ud v\]
which lies in $L^2(\R)$.
Therefore the Lebesgue theorem tells us that 
\[
\ds\lim_{\eps,\eps'\to 0}\Delta_{\eps,\eps'}=0.\]

We go back to $D_{x,1}$ that we split into
\[
D_{x,1}(t)=\ds\int_0^{t-\eta}...\ud s+\ds\int_{t-\eta}^t...\ud s
\]
with $0<\eta\ll t\leq T<\infty$ to be determined.
The  integral on $(0,t-\eta)$ can be estimated owing to the previous manipulations and the Cauchy--Schwarz inequality; we get
\[\begin{array}{lll}
\left|\ds\int_0^{t-\eta}...\ud s\right|
&\leq& \|\rho^{(\eps)}\|_{L^\infty}
\ds\int_0^{t-\eta} \ds\frac{C_0}{\sqrt{t-s}} \ds\sup_{x',y'}| \mathcal I_{\eps,\eps'}(x',y',s) |\ud s
\\
&\leq&
\ds\frac{C_0\sqrt\pi}{2}\|\rho^{(\eps)}\|_{L^\infty}\sqrt{\Delta_{\eps,\eps'}}
\left(\ds\int_0^{t-\eta}\ds\frac{\ud s}{t-s}\right)^{1/2}\left(\ds\int_0^{t-\eta}\ds\iint 
(1+|z|^2)|\nabla \rho^{(\eps)}(z,s)|^2\ud s
\ud z\right)^{1/2}
\\
&\leq&
C_T\sqrt{\Delta_{\eps,\eps'}}
\sqrt{\ln(t/\eta)}
\end{array}
\]
for a certain $C_T>0$, that comes from the estimates in Lemma~\ref{west}.
For the integral over $(t-\eta,t)$, we claim that we can find a constant, still denoted $C_T>0$, such that
\begin{eqnarray}
\left|\ds\int_{t-\eta}^t...\ud s\right|
&\leq& \int_{t-\eta}^t \frac{C_0}{\sqrt{t-s}}\iiiint \left[\delta^{(\eps)}(y'-y")+\delta^{(\eps')}(y'-y")\right]\\
&&\qquad\qquad\qquad\times
\rho^{(\eps)}(x',y',s)\rho^{(\eps')}(x",y",s)\ud x'\ud y\ud x"\ud y"\ud s \nonumber \\
&\leq &C_T\sqrt{\eta}. \nonumber
\end{eqnarray}
This conclusion follows from uniform bounds  (with respect to $\eps,\eps'$ and $s$) of expressions like
\[\mathcal J_{\eps,\eps'}(s)=
\int \delta^{(\eps)}(y'-y")\rho^{(\eps)}(x',y')\rho^{(\eps')}(x",y")\ud z'\ud z".
\]
Let us set $$\tilde{\rho}^{(\eps)}(x",y',s)=\int \delta^{(\eps)}(y'-y")\rho^{(\eps)}(x",y",s)\ud y".$$
We control $\mathcal J_{\eps,\eps'}(s)$ by using moments. Indeed, we get
\begin{eqnarray}
\mathcal J_{\eps,\eps'}(s) &=& \iiint \rho^{(\eps)}(x',y',s)\tilde{\rho}^{(\eps')}(x",y',s)\ud x'\ud y'\ud x" \nonumber \\
&\leq & \frac12 \iiint \frac{1+x'^2}{1+x"^2} |\rho^{(\eps)}|^2(x',y',s)\ud x'\ud y'\ud x" \nonumber
\\
&&\quad+ \frac12 \iiiint \frac{1+x"^2}{1+x'^2} |\tilde{\rho}^{(\eps')}|^2(x",y',s)\ud x'\ud y'\ud x"\nonumber \\
&\leq & \frac{\pi}{2} \iint (1+x^2)|\rho^{(\eps)}|^2(x,y,s)\ud x\ud y+\frac{\pi}{2}\iint(1+x^2)|\tilde{\rho}^{(\eps')}|^2(x,y,s)\ud x\ud y.\nonumber
\end{eqnarray}
Owing to Lemma~\ref{west}, we already know that the first integral in the right hand side is bounded (the constant depends on the final time). For the second term, 
we simply write
\[\begin{array}{l}
\ds\iint(1+x^2)|\tilde{\rho}^{(\eps')}|^2(x,y,s)\ud x\ud y
\\
\qquad
\leq 
\ds\iint(1+x^2)\left|
\int \sqrt{\delta^{(\eps)}(y-y')}\  \sqrt{\delta^{(\eps)}(y-y')}\rho^{(\eps)}(x,y',s)\ud y'
\right|^2(x,y,s)\ud x\ud y
\\
\qquad
\leq 
\ds\iint(1+x^2)\left\{ 
\int \delta^{(\eps)}(y-y')\ud y' \times\int\delta^{(\eps)}(y-y')|\rho^{(\eps)}(x,y',s)|^2\ud y'
\right\}\ud x\ud y
\\
\qquad\leq 
\ds\iint(1+x^2) |\rho^{(\eps)}(x,y',s)|^2 \left(\ds\int
\delta^{(\eps)}(y-y')\ud y
\right)\ud x\ud y'
\\
\qquad\leq 
\ds\iint(1+x^2) |\rho^{(\eps)}(x,y',s)|^2\ud x\ud y'
\end{array}\]
which is thus also bounded uniformly with respect to $\eps,\eps'>0$ and $0\leq s\leq T<\infty$.
Finally, we arrive at
\[
|D_{x,1}(t)|\leq C_T\left(\sqrt{\ln( t/\eta)} \sqrt{\Delta_{\eps,\eps'}}+ \sqrt{\eta}\right)
\] which holds for any $0<\eta\ll t\leq T<\infty$.
It shows that  $\lim_{(\eps,\eps')\to 0}D_{x,1}(t)=0$ uniformly on $[0,T]$.\\

The analysis of $D_{x,2}$ is simpler; it relies on the following observation
\[\begin{array}{l}
\ds\left|\iint \delta^{(\eps')}(y'-y")\big(\sgn^{(\eps)}(x'-x")-\sgn^{(\eps')}(x'-x")\big)\rho^{(\eps')}(x",y",s)\ud x"\ud y"
\right| 
 \nonumber \\
\qquad\qquad\ds\leq \|\rho^{(\eps')}\|_\infty \int |\sgn^{(\eps)}(x'-x")-\sgn^{(\eps')}(x'-x")|\ud x" \\
\qquad\qquad\ds
\leq \sqrt{\ds\frac2\pi}
\ds\int\Big| \ds\int_{u/\eps'}^{u/\eps} e^{-v^2/2}\ud v
 \Big|\ud u= \tilde\Delta_{\eps,\eps'}.
\end{array}\]
A straightforward adaptation of the argument used for studying $\Delta_{\eps,\eps'}$ shows that $\lim_{\eps,\eps'\to 0} \tilde\Delta_{\eps,\eps'}=0$ and we have
\[
|D_{x,2}(t)| \leq \int_0^t \frac{C_0}{\sqrt{t-s}} \|\rho^{(\eps)}(s,\cdot)\|_{L^1}\ \tilde \Delta_{\eps,\eps'}\ud s \leq C_T \tilde\Delta_{\eps,\eps'}
\]
for any $0\leq t\leq T<\infty$.
Of course, $D_{y,1}$ and $D_{y,2}$ can be dealt with in a similar manner.
\qed

Coming back to \eqref{int112}, 
we arrive at 
\begin{equation}
\label{est_important}
\|(\rho^{(\eps)}-\rho^{(\eps')})(t,\cdot)\|_{L^1}
\leq 
\big(\mathcal A_{\eps,\eps'} + \widetilde{\mathcal A}(R,\lambda)\big) + \mathcal B(R)\ds\int_0^t\ds\frac{\|(\rho^{(\eps)}-\rho^{(\eps')})(s,\cdot)\|_{L^1}}{\sqrt{t-s}}\ud s,
\end{equation}
which holds for any $0\leq t\leq T<\infty$ and $0<R<\infty$  with 
\begin{equation}
\label{est_important2}
\begin{array}{l}
\mathcal A_{\eps,\eps'}=\ds\sup_{0\leq t\leq T} A_{\eps,\eps'}(t)+ \ds\sup_{0\leq t\leq T}D_{\eps,\eps'}(t),
\\
 \widetilde{\mathcal A}(R,\lambda)=
 (\beta_2+\gamma_2)\ds\frac{1+\lambda}{\lambda^2} \sqrt Te^{\varphi(\lambda)T} e^{-\alpha\lambda R},
\\
\mathcal B(R)= (\beta_1+\gamma_1)RM,
\end{array}
\end{equation}
with $M=\sup_{\eps>0}\|\rho^{(\eps)}\|_{L^\infty}$, which is known to be finite.
We should bear in mind the fact that $\beta_2$ and $\gamma_2$ depend on $\lambda$ too, through the exponential moments
$\mathcal E_0(2\lambda(1+\alpha))$.
Applying the singular Gr\"onwall Lemma \ref{sing_G} leads to 
\[
\|(\rho^{(\eps)}-\rho^{(\eps')})(t,\cdot)\|_{L^1}\leq \big(\mathcal A_{\eps,\eps'}+ \widetilde{\mathcal A}(R,\lambda) \big)E_{1/2}\Big(\ds\frac{\mathcal B(R)}{2}\sqrt t\Big).
\]
We remind the reader that the Mittag--Leffler function is explicitely known
\[
E_{1/2}(z)=\ds\sum_{k=0}^\infty \ds\frac{z^k}{\Gamma(1+k/2)}=e^{z^2}\mathrm{erfc}(-z)=\ds\frac{2}{\sqrt\pi} \ e^{z^2} \ds\int_{-\infty}^z e^{-u^2}\ud u.
\]

We are paying attention to the term 
$\widetilde{\mathcal A}(R,\lambda)E_{1/2}\Big(\ds\frac{\mathcal B(R)}{2}\sqrt t\Big)$.
This is where we make use of \eqref{hypci} to control $\mathcal E_0(2\lambda(1+\alpha))$ in the coefficients $\beta_2,\gamma_2$.
As far as $\lambda \geq 1$, we have
$\varphi(\lambda)\leq 4D(1+\alpha)^2 \lambda^2$.
Therefore, 
up to some irrelevant constant hereafter denoted by $K>0$, 
the quantity of interest can be dominated by
\[\ds\frac{\sqrt T}{\lambda}\exp\Big((DT +p_2)4(1+\alpha)^2 \lambda^2  -\alpha R\lambda + p_0 +\ds\frac{ (\beta_1+\gamma_1)^2M^2}{4} TR^2 \Big).
\] 
The exponent recasts as
\[
4(DT +p_2)(1+\alpha)^2\Big(\lambda - \ds\frac{\alpha R}{8(DT +p_2)(1+\alpha)^2} \Big)^2
- R^2\Big(
\ds\frac{\alpha^2}{16(DT +p_2)(1+\alpha)^2}- \ds\frac{ (\beta_1+\gamma_1)^2M^2}{4} T
\Big) + p_0.
\] 
We start by picking $0<T<T_{\star}$ small enough, so that 
$$
\ds\frac{\alpha^2}{16(DT +p_2)(1+\alpha)^2}- \ds\frac{ (\beta_1+\gamma_1)^2M^2}{4} T
\geq \ds\frac{\alpha^2}{8p_2(1+\alpha)^2}=q_2>0$$
holds for any $0\leq t\leq T_\star$.
Next, let $\omega>0$.
We can find $R=R(\omega)$ large enough
so that 
\[
K \sqrt T\ e^{p_0} e^{-R^2q_2}\leq \ds\frac\omega 2
\] 
holds.
Possibly enlarging $R(\omega)$, we also suppose that 
$$\ds\frac{\alpha R}{8(DT +p_2)(1+\alpha)^2}\geq 1.$$
We then make use of the estimates with 
$$\lambda= \ds\frac{\alpha R}{8(DT +p_2)(1+\alpha)^2}$$
which leads to 
\[\widetilde{\mathcal A}(R,\lambda)E_{1/2}\Big(\ds\frac{\mathcal B(R)}{2}\sqrt t\Big)\leq \ds\frac\omega2.\]
Finally, there exists $\eps(\omega)>0$ small enough such that for any $0<\eps,\eps'\leq \eps(\omega)$ 
we get 
\[
\mathcal A_{\eps,\eps'}
E_{1/2}\Big(\ds\frac{\mathcal B(R)}{2}\sqrt t\Big)\leq \ds\frac\omega2.
\]
It follows that 
\[
\|(\rho^{(\eps)}-\rho^{(\eps')})(t,\cdot)\|_{L^1}\leq \omega
\]
holds for any $0\leq t\leq T\leq T_\star$, provided $0<\eps,\eps'\leq \eps(\omega)$.
We extend this result on any time interval  by repeating the reasoning on subintervals of length smaller than $T_\star$. 
Therefore $\big(\rho^{(\eps)}\big)_{\eps>0}$ is a Cauchy sequence in the Banach space $C([0,T],L^1(\mathbb R^2))$ and it converges strongly to a solution of \eqref{pde}--\eqref{ci}.
The proof can be readily adapted to establish the uniqueness of the solution of  \eqref{pde}--\eqref{ci} for a symmetric initial data verifying \eqref{hypci}.
\qed

\subsubsection{A convergence rate for $\rho^{(\eps)}$}

Following the same strategy as in the proof of Theorem \ref{main2}, it is possible to give a rate of convergence for $\rho^{(\eps)}$.
\begin{theorem}\label{rate}
Let $T$ be a fixed time. Let $\rho^{(\eps)}$ be the symmetric solution of \eqref{reg_pde}--\eqref{approxF} with initial data $\rho_0$ ($\rho_0$ is assumed to be symmetric), and let $\rho$ be the symmetric solution of \eqref{pde}--\eqref{def_f} with same initial data $\rho_0$. Then there exist constants $C(\rho_0,T)$ and $0<\nu(\rho_0,T)<1$ depending on both $\rho_0$ and $T$, such that
\begin{equation}
\sup_{t\in [0,T]}||(\rho^{(\eps)}-\rho)(t)||_{L^1} \leq  C(\rho_0,T)  \eps^{\frac12 \nu(\rho_0,T)}
\end{equation}
\end{theorem}

\begin{remark}
Observe
that $\nu(\rho_0,T)$ is always smaller than $1$, and
it has the following asymptotic behavior
\[
\ds\lim_{T\to 0} \nu(\rho_0,T)=1,\qquad
\ds\lim_{T\to +\infty} \nu(\rho_0,T)=0,\]
 for any $\rho_0$.
Note the $1/2$ factor: with the present proof the convergence rate cannot be better than $\eps^{1/2}$.
\end{remark}

\noindent
{\bf Proof.}
The idea is to revisit the computations in Section~\ref{sec:cauchy}, in order to estimate more accurately the distance between $\rho^{(\eps)}$ and $\rho$, solution of the singular PDE. Since we have used  estimates that are uniform with respect to  $\eps$, we may simply take $\eps'=0$ in the computations performed above.
It leads to the following observations:

\begin{itemize}
\item \emph{$A$ term:} We take the same initial condition for $\rho^{(\eps)}$ and $\rho$, hence the error related to the initial condition simply vanishes: $A_\eps(t)=0$.\\

\item
\emph{$B$ and $C$ terms:} We use Lemmas \ref{l2} and \ref{l3}, with $\eps'=0$.\\

\item
\emph{$D$ term:} We need to estimate
\[
\Delta_\eps = \frac{4}{\pi} \int_0^\infty \left| \int_{u/\eps}^\infty e^{-v^2/2}\ud v\right|^2 \ud u.
\] 
Since $v^2/2\geq x^2/2+x(v-x)$, we get
\[
\int_x^\infty e^{-v^2/2}\ud v \leq e^{-x^2/2} \int_0^\infty e^{-xs}\ud s \leq \frac{e^{-x^2/2}}{x}.
\]
Thus, for any $\alpha>0$ we obtain 
\begin{eqnarray}
\Delta_\eps &\leq&  \frac{4}{\pi} \left(\int_0^\alpha \sqrt{\frac{\pi}{2}}\ud u + \int_\alpha^\infty \frac{\eps^2}{u^2}e^{-u^2/\eps^2}\ud u\right) \nonumber \\
&\leq &\frac{4}{\pi} \left( \alpha \sqrt{\frac{\pi}{2}} + \eps \int_{\alpha/\eps}^\infty \frac{1}{s^2}e^{-s^2}\ud s\right) .\nonumber 
\end{eqnarray}
Choosing $\alpha=\eps$, this relation   yields 
$$\Delta_\eps 
 \leq  C\eps $$
where $C$ is an absolute constant.
A very similar reasoning applied to 
\[
\tilde{\Delta}_\eps = 2\sqrt{\frac{2}{\pi}} \int_0^\infty \left| \int_{u/\eps}^\infty e^{-v^2/2}\ud v\right| \ud u
\]
yields
\[
\tilde{\Delta}_\eps \leq C \eps
\]
where again $C$ is an absolute constant.

The estimate for $D_{x,1}(t)$ reads, for any $0<\eta<t$,
\[
|D_{x,1}(t)| \leq C_T \left(\sqrt{\Delta_\eps}\sqrt{\ln(t/\eta)}+\sqrt{\eta}\right).
\]
Choosing $\eta=\Delta_\eps$ (it is possible to do marginally better), we obtain, at the price of modifying $C_T$,
\[
|D_{x,1}(t)| \leq C_T \sqrt{\Delta_\eps}\sqrt{\ln(t/\Delta_\eps)}~,
\]
which, according to the above estimate for $\Delta_\eps$, yields
\[
|D_{x,1}(t)| \leq C_T \sqrt{\eps}\sqrt{\ln(t/\eps)}.
\]
Since $D_{x,2}\leq C_T \eps$, we see that $D_{x,1}$ is the largest contribution to $D_\eps$.
\end{itemize}

We use now \eqref{est_important}--\eqref{est_important2} of the previous section with $\eps'=0$:
\begin{eqnarray}
||(\rho^{(\eps)}-\rho)(t)||_{L^1} &\leq &(\mathcal{A}_\eps+\tilde{A}(R,\lambda))E_{1/2}\left(\frac{(\beta_1+\gamma_1)^2M^2R^2T}{4}\right). \label{eq:est1}
\end{eqnarray}
The contribution to $\mathcal{A}_\eps$ coming from the initial condition vanishes, since we choose the same initial condition for $\rho^{(\eps)}$ and 
$\rho$. The second contribution to  $\mathcal{A}_\eps$ comes from the "$D$ terms", which are smaller than $C_T \sqrt{\eps}\sqrt{\ln(T/\eps)}$.

We can play the same game as in the proof of the Cauchy property: write $E_{1/2}(z)\leq c e^{z^2}$ for some $c$, and observe that  the exponent in \eqref{eq:est1} can be rewritten as
\begin{equation}
\label{expo}
4(DT+p_2)(1+\alpha)^2\left(\lambda -\frac{\alpha R}{8(DT+p_2(1+\alpha)^2}\right)^2-R^2\left(\frac{\alpha^2}{4(DT+p_2)(1+\alpha)^2}-\frac{(\beta_1+\gamma_1)^2M^2 T}{4}\right)+p_0.
\end{equation}
We choose $T=T^\ast$ small enough so that the second term, proportional to $R^2$ is negative, which means
\[
\frac{\alpha^2}{4(DT^\ast+p_2)(1+\alpha)^2}-\frac{(\beta_1+\gamma_1)^2M^2 T^\ast}{4} =q_2(T^\ast)>0
\]
Then we choose $\lambda$ such that the first term in \eqref{expo} vanishes.
We finally obtain
\begin{equation}
\label{rhsR}
\sup_{t\in [0,T^\ast]}||(\rho^{(\eps)}-\rho)(t)||_{L^1} \leq  C\mathcal{A}_\eps \exp{(K_{T^\ast} R^2)} + C' \exp{(-q_2 R^2)}
\end{equation} 
where $C$ and $C'$ depend on $T^\ast$, and 
\[
K_T = \frac{(\beta_1+\gamma_1)^2 M^2 T}{4}.
\]  
We now choose $R$ to minimize the right hand side of \eqref{rhsR}. For instance, taking $R$ such that
\[
\exp{[-(K_{T^\ast}+q_2(T^\ast))R^2]} = \mathcal{A}_\eps
\]
yields, for a modified $C$,
\begin{equation}
\sup_{t\in [0,T^\ast]}||(\rho^{(\eps)}-\rho)(t)||_{L^1} \leq  C  \left[\sqrt{\eps}\sqrt{\ln(T^\ast/\eps)} \right]^{\bar{\nu}}.
\end{equation}
with 
\[
\bar{\nu} = \frac{q_2(T^\ast)}{K_{T^\ast}+q_2(T^\ast)}
\]
Slightly decreasing $\bar{\nu}$ to absorb the logarithmic term, this proves the claim for any $T<T^\ast$. For $T>T^\ast$, we 
divide $[0,T]$ into subintervals of size $T^\ast$, and apply the previous strategy for each subinterval.
We have to take into account the error related to initial condition at the beginning of each subinterval.  This error is given by the total error at the end of the previous 
subinterval. Thus we have to reintroduce an error related to initial data. Calling $E_k$ the bound on the error at the end of the interval $[(k-1)T^\ast, kT^\ast]$, 
and $\mathcal{A}_\eps^{(k)}$ the $\mathcal{A}_\eps$ term to be considered on the interval  $[kT^\ast, (k+1)T^\ast]$, we have
\[
\mathcal{A}_\eps^{(k)} \leq C_{T^\ast} \sqrt{\eps}\sqrt{\ln(T^\ast/\eps)} +  C_{T^\ast} E_k \leq C_{T^\ast} E_k, 
\]
where $C_{T^\ast}$ can take different values, but remains a constant depending on $\rho_0, T^\ast$, and not on $\eps$.
With the same reasoning as above, we conclude with
\[
E_{k+1} \leq C_{T^\ast} E_k^{\bar{\nu}}.
\]
Since $T^\ast$ is of order $1$, we have to repeat the argument on a finite number of subintervals to reach the prescribed time $T$. Each iteration of course decreases the convergence rate, and increases the prefactor, but for any $T$, we can guarantee a finite $\nu$, as claimed.
\qed

\section{Particle approximation}\label{sec:particles}

We consider now an $N$-particle description of the dynamics.
Namely, let $Z^{(\eps)}_i=(X^{(\eps)}_i,Y^{(\eps)}_i)$ be the solution of the stochastic differential system 
\begin{eqnarray}
\ud X^{(\eps)}_{i,t}&=& \frac1N \sum_{j\neq i} K_x^{(\eps)}(Z^{(\eps)}_{i,t}-Z^{(\eps)}_{j,t}) \ud t +\sqrt{2D}\ud B_{i,x,t}, 
\label{eq:reg_part_x}\\
\ud Y^{(\eps)}_{i,t} &=& \frac1N \sum_{j\neq i} K_y^{(\eps)}(Z^{(\eps)}_{i,t}-Z^{(\eps)}_{j,t} ) \ud t +\sqrt{2D} \ud B_{i,y,t}, 
\label{eq:reg_part_y}
\end{eqnarray}
where $B_{i,x}$ and $B_{i,y}$ are independent Brownian motions.
Here and below, the interaction kernel is given by 
\begin{eqnarray}
K_x^{(\eps)}(z) &=& -\mbox{sgn}^{(\eps)}(x) \delta^{(\eps)}(y) \nonumber \\
K_y^{(\eps)}(z) &=& -\mbox{sgn}^{(\eps)}(y) \delta^{(\eps)}(x), \nonumber
\end{eqnarray}
with $z=(x,y)$.
It is then clear that $\|  K_x^{(\eps)} \|_{\mathrm{Lip}} =C/\varepsilon^2$, and the same holds true for  $K_y^{(\eps)}$.
We assume that the initial conditions for the particles' trajectories  $$Z^{(\eps)}_{i,t}\Big|_{t=0}=Z^{(\eps)}_{i,0}$$
 are independent random variables, with common law $\rho_0$.
 In the discussion, we naturally assume 
that $\rho_0$ is a probability density.
Accordingly, for both $\rho$ and $\rho^{(\varepsilon)}$
solutions of \eqref{pde} and \eqref{reg_pde} respectively, associated to the initial data $\rho_0$, we have
\[\ds\iint \rho\ud z=\ds\iint \rho^{(\varepsilon)}\ud z=\ds\iint \rho_0\ud z=1.\]
Moreover, we assume throughout this section that $\rho_{0}$ is such that the symmetric existence theorem works
as we shall use the rate of convergence established in this framework.
We associate to the solutions of this system \eqref{eq:reg_part_x}--\eqref{eq:reg_part_y}, the empirical measure
 \[
\widehat {\rho}^{(\eps),N} = \frac1N \sum_{i=1}^N \delta(z-Z^{(\eps)}_i).
\]
Note that the interaction force in \eqref{eq:reg_part_x}--\eqref{eq:reg_part_y} has been rescaled by the $1/N$ factor (roughly speaking we have replaced the kernel $K^{(\eps)}$ by $\frac1N K^{(\eps)}$), 
so that the total force exerted on a given particle remains of order 1; this is the so--called \emph{mean field regime}.
We refer the reader to the surveys \cite{Boll, FGol} for an introduction to such regimes.
The goal of this section is to investigate the convergence of this particle approximation to $\rho$, the solution of the singular PDE \eqref{pde}
in the regime $N\to \infty$, $\eps\to 0$. 
\\

The analysis uses the Wasserstein distance, see  \cite{Dob,CV} for a thorough discussion on this notion. The Wasserstein distance $W_1(\mu,\nu)$ between two probability measures $\mu, \nu$ on $\R^2$ is defined as 
\[
W_1(\mu,\nu)=\sup\left\{\left|\int\varphi(z)\mu(\mathrm d z)-\int\varphi(z)\nu(\mathrm d z)\right| ,
\|\varphi\|_{\mathrm{Lip}}\le 1 \right\},
\]
where 
\[\|\varphi\|_{\mathrm{Lip}}= \ds\sup_{x\neq y,\ \ x,y\in\R^2}\ds\frac{|\varphi(x)-\varphi(y)|}{|x-y|}.
\]
 Note that $W_1$ determines the topology of tight convergence on the space of probability measures on $\mathbb R^2$, see \cite[Chap. 6]{CV}. 
 
 Wasserstein metric is well defined on the set of probability measures with finite first moment. This is the case for $\rho$, the solution of the original PDE \eqref{pde}, see Proposition  \ref{apriori} as well as $\rho^{(\eps)}$, the solution of the regularized PDE \eqref{reg_pde}, see Proposition \ref{lemma:rho_eps_estimates}. It also holds true for the particle approximations $\widehat {\rho}^{(\eps),N}$ (they are finite sums of Dirac delta distributions).

It turns out that $W_1$ is a well adapted tool to investigate the limit $N\to\infty$, see \cite{Boll, Dob, FGol, Sz}.
The strategy is to write 
\[
W_1(\widehat {\rho}^{(\eps),N},\rho) \leq W_1(\widehat {\rho}^{(\eps),N},\rho^{(\eps)}) + W_1(\rho^{(\eps)},\rho),
\]
where $\rho^{(\eps)}$ is the solution of the regularized PDE \eqref{reg_pde}.
The second term is controlled by the rate of convergence established in  the previous  section, and the first one by adapting ``standard'' MacKean--Vlasov estimates, as we are going to detail now.
According to \cite{Sz}, 
we start by introducing an auxiliary system of interacting particles.
The solution $\rho^{(\eps)}$ of the regularized PDE  is also the law of the solution of the system of SDE
\begin{eqnarray}
\ud\tilde{X}^{(\eps)}_{i} &=& (K_x^{(\eps)} \star \rho^{(\eps)})(\tilde{Z}^{(\eps)}_i) \ud t +\sqrt{2D}\ud B_{i,x} ,\label{eq:MKx}\\
\ud\tilde{Y}^{(\eps)}_{i} &=& (K_y^{(\eps)} \star \rho^{(\eps)})(\tilde{Z}^{(\eps)}_i)\ud t +\sqrt{2D}\ud B_{i,y}. 
\label{eq:MKy}
\end{eqnarray}
Note that both $Z^{(\eps)}_i=(X^{(\eps)}_i,Y^{(\eps)}_i)$ and $\tilde Z^{(\eps)}_i=(\tilde X^{(\eps)}_i,\tilde Y^{(\eps)}_i)$ are driven by the same Brownian motions and we choose them to have the same initial condition. The system of stochastic differential equations \eqref{eq:MKx}--\eqref{eq:MKy} (respectively \eqref{eq:reg_part_x}--\eqref{eq:reg_part_y}) has a unique solution, as the coefficients $(t,z)\mapsto K_x^{(\eps)} \star \rho^{(\eps)}(z)$ and $(t,z)\mapsto K_y^{(\eps)} \star \rho^{(\eps)}(z)$ are Lipschitz with respect to $z$ and continuous with respect to $t$.  Moreover, 
the law $\mu^{(\eps)}$ of $\tilde Z^{(\eps)}_i=(\tilde X^{(\eps)}_i,\tilde Y^{(\eps)}_i)$ is a (weak) solution of
\begin{eqnarray}
\partial_t \mu^{(\eps)} &=& \nabla\cdot \left( (-K^{(\eps)} \star \rho^{(\eps)}) \mu^{(\eps)}\right) +D \Delta \mu^{(\eps)} \nonumber \\
\mu^{(\eps)}\Big|_{t=0} &=&\rho_0.\nonumber
\end{eqnarray}
Since this equation has a unique solution, and $\rho^{(\eps)}$ is a solution, it follows that $\mu^{(\eps)}=\rho^{(\eps)}$.  We define $\widetilde{\rho}^{(\eps),N}$ to be the empirical measure associated with the $\tilde{Z}^{(\eps)}_i$:
\[
\widetilde{\rho}^{(\eps),N} = \frac1N \sum_{i=1}^N \delta(z-\tilde{Z}^{(\eps)}_i).
\]
The following statement is an immediate corollary of Theorem 1 in \cite{FG}: 

\begin{prop}\label{rateindparticles} Let $0<T<\infty$.
Assume that there exist $q>2$ and a constant $C=C(T)$, that depends on $T$ but is independent of $\varepsilon$, such that 
\begin{equation}\label{morethantwomoment}
\sup_{t\in [0,T]}\int|z|^q\rho^{(\eps)}(\mathrm d z)\le C.
\end{equation}
Then there exists a constant $\tilde C= \tilde C(T)$ independent of $\varepsilon$ such that   
\begin{equation}\label{tb}
\sup_{t\in [0,T]}{\mathbb E}[W_1(\widetilde{\rho}^{(\eps),N},\rho^{(\eps)})]\le {\tilde C\over \sqrt{N}}\log{(1+N)}.
\end{equation}
\end{prop}

The uniform bound \eqref{morethantwomoment} holds true if, for example, the initial measure $\rho_0$ has a finite third moment. To prove this, one uses an  argument similar to that in Proposition \ref{lemma:rho_eps_estimates}--$iv$ (in effect one uses the same proof as that used for the a priori bound deduced for the original measure $\rho$ in  Proposition \ref{apriori}--$iv$).
We state and prove now the main result of the section.
\begin{theorem} \label{ratep}
Let $0<T<\infty$ be a fixed time. Under the same conditions as in Theorem \ref{rate}, we have 
\begin{equation}\label{particlerate}
\sup_{t\in [0,T]}{\mathbb E}[W_1(\widehat {\rho}^{(\eps),N},\rho)]\le {\tilde C e^{{2C             T\over \eps^2}}\over \sqrt{N}}\log{(1+N)}+ C_\rho  \eps^{\frac12 \nu_{\rho}},
\end{equation}
where  $\tilde C= \tilde C(T)$ is the constant defined in Proposition \ref{rateindparticles},  $C$ is the Lipschitz constant of $\varepsilon^2 K^{(\eps)}$
and  $C_\rho=C(\rho_0,T)$, respectively, $\nu_{\rho}=\nu(\rho_0,T)$ are the constants arising from Theorem \ref{rate}. 

In particular,  for any $\delta\in [0,{1\over 2})$, there exists $\eps=\eps(\delta,N)$ and a constant $\tilde C_\rho=\tilde C_\rho(\delta)$ independent of $N$\ such that
 \begin{equation}\label{lograte}
 \sup_{t\in [0,T]}{\mathbb E}[W_1(\widehat {\rho}^{(\eps),N},\rho)]\le \tilde C_\rho \left(\log(N)\right)^{-{1\over 4}\nu_{\rho}}\ 
 \end{equation}
 for any $N\ge 1$.
\end{theorem}
\begin{proof} Following  Theorem \ref{rate}, to establish \eqref{particlerate} it suffices to prove that 
\begin{equation}\label{nnn5}
\sup_{t\in [0,T]}{\mathbb E}[W_1(\widehat {\rho}^{(\eps),N},\rho^{(\eps)})]\le {\tilde C e^{{2CT\over \eps^2}}\over \sqrt{N}}\log{(1+N)}.
\end{equation}
Since both $Z^{(\eps)}_i=(X^{(\eps)}_i,Y^{(\eps)}_i)$ and $\tilde Z^{(\eps)}_i=(\tilde X^{(\eps)}_i,\tilde Y^{(\eps)}_i)$ are driven by the same Brownian motions and have the same initial condition, we have 
\begin{eqnarray}
{\ud \over \ud t}(Z^{(\eps)}_{i,t}-\tilde{Z}^{(\eps)}_{i,t})&=& \frac1N \sum_j K^{(\eps)} (Z^{(\eps)}_{i,t}-Z_{j,t}^{(\eps)}) - (K^{(\eps)}\star \rho^{(\eps)})(\tilde{Z}^{(\eps)}_{i,t}) \nonumber \\
&=& (K^{(\eps)} \star \widehat {\rho}^{(\eps),N})(Z^{(\eps)}_{i,t}) - (K^{(\eps)}\star \rho^\eps)(\tilde{Z}^{(\eps)}_{i,t}) \nonumber \\
&=& [K^{(\eps)} \star (\widehat {\rho}^{(\eps),N}-\rho^{(\eps)})(Z^{(\eps)}_{i,t})] +[ (K^{(\eps)}\star \rho^{(\eps)})(Z^{(\eps)}_{i,t})-(K^{(\eps)}\star \rho^{(\eps)})(\tilde{Z}^{(\eps)}_{i,t})].\nonumber
\end{eqnarray}
We note that $ K^{(\eps)}(\cdot-Z^{(\eps)}_{i,t})$ is a function with Lipschitz constant less than ${C\over \eps^2}$.
Hence
\[
\big|K^{(\eps)} \star (\widehat {\rho}^{(\eps),N}-\rho^{(\eps)})(Z^{(\eps)}_{i,t})\big| \leq \frac{C}{\eps^2}W_1(\widehat {\rho}^{(\eps),N},\rho^{(\eps)}).
\]
Furthermore, using that $K^{(\eps)}(z-\cdot)$ is $C/\eps^2$-Lipschitz, and $\int \rho^\eps\ud z=1$, we get
\begin{eqnarray}
\big| K^{(\eps)}\star \rho^{(\eps)}(Z^{(\eps)}_{i,t})-K^{(\eps)}\star \rho^{(\eps)}(\tilde{Z}^{(\eps)}_{i,t})\big|  \leq \frac{C}{\eps^2}|Z^{(\eps)}_{i,t}-\tilde{Z}^{(\eps)}_{i,t}|. 
\end{eqnarray}
Then
\[
{\ud \over \ud t}|Z^{(\eps)}_{i,t}-\tilde{Z}^{(\eps)}_{i,t}| \leq \frac{C}{\eps^2}|Z^{(\eps)}_{i,t}-\tilde{Z}^{(\eps)}_{i,t}| + \frac{C}{\eps^2}W_1(\widehat {\rho}^{(\eps),N},\rho^{(\eps)}).
\]
Hence, since $Z^{(\eps)}$ and $\tilde Z^{(\eps)}$ share the same initial data, we arrive at
\begin{equation}
e^{-{Ct\over \eps^2}}|Z^{(\eps)}_{i,t}-\tilde{Z}^{(\eps)}_{i,t}|\le \int_0^t \frac{Ce^{-{Cs\over \eps^2}}}{\eps^2}W_1(\widehat {\rho}^{(\eps),N},\rho^{(\eps)})(s)\ud s.
\label{eq:bound}
\end{equation}
Now we write
\begin{eqnarray}
e^{-{Ct\over \eps^2}} W_1(\widehat \rho^{(\eps),N},\rho^{(\eps)}) &\leq & e^{-{Ct\over \eps^2}}W_1(\widehat {\rho}^{(\eps),N},\widetilde{\rho}^{(\eps),N}) + e^{-{Ct\over \eps^2}}W_1(\widetilde {\rho}^{(\eps),N},\rho^{(\eps)}) \nonumber \\
&\leq& \frac1N \ds\sum_{i=1}^N e^{-{Ct\over \eps^2}}|Z^{(\eps)}_{i,t}-\tilde{Z}^{(\eps)}_{i,t}| +  e^{-{Ct\over \eps^2}}W_1(\widetilde {\rho}^{(\eps),N},\rho^{(\eps)}) \nonumber \\
&\leq & \int_0^t \frac{Ce^{-{Cs\over \eps^2}}}{\eps^2}W_1(\widehat {\rho}^{(\eps),N},\rho^{(\eps)})(s)\ud s +  e^{-{Ct\over \eps^2}}W_1(\widetilde {\rho}^{(\eps),N},\rho^{(\eps)}), \nonumber
\end{eqnarray}
where we have used first the triangle inequality, then a direct inequality for the $W_1$ distance between the two empirical measures, and finally  \eqref{eq:bound}.
By taking the expectation and using \eqref{tb}, we obtain   
\begin{eqnarray}
e^{-{Ct\over \eps^2}}{\mathbb E}W_1(\widehat \rho^{(\eps),N},\rho^{(\eps)})(t) &\leq & (1-e^{-{Ct\over \eps^2}}) {\tilde C\over \sqrt{N}}\log{(1+N)}+\frac{C}{\eps^2}
\int_0^t e^{-{Cs\over \eps^2}}{\mathbb E}W_1(\widehat \rho^{(\eps),N},\rho^{(\eps)})(s)\ud s.\nonumber\\
&\leq &  {\tilde C\over \sqrt{N}}\log{(1+N)}+\frac{C}{\eps^2}
\int_0^t e^{-{Cs\over \eps^2}}{\mathbb E}W_1(\widehat \rho^{(\eps),N},\rho^{(\eps)})(s)\ud s.\nonumber
\end{eqnarray}
By the standard Gr\"onwall's lemma we deduce that 
\[
e^{-{Ct\over \eps^2}}{\mathbb E}W_1(\widehat \rho^{(\eps),N},\rho^{(\eps)})(t) \leq {\tilde Ce^{{Ct\over \eps^2}}\over \sqrt{N}}\log{(1+N)}
\]
which gives \eqref{nnn5}. Using the triangle inequality and Theorem~\ref{rate},  \eqref{nnn5}  leads to \eqref{particlerate}.
Moreover observe that for $\delta\in [0,{1\over 2})$ and $\eps=({1-2\delta\over 4CT}\log(N))^{-{1\over 2}}$ we have 
\[
{\tilde C e^{{2CT\over \eps^2}}\over \sqrt{N}}\log{(1+N)}+ C_\rho  \eps^{\frac12 \nu_{\rho}}={\tilde C \over N^\delta}\log{(1+N)}+ C_\rho \left({1-2\delta\over 4CT}\log(N)\right)^{-{1\over 4}\nu_{\rho}}
\] 
which gives \eqref{lograte}.
\end{proof}

\section{Numerical illustrations}
\label{Sec:Num}

The goal of this section is two--fold:
\begin{enumerate}
\item Illustrate the existence Theorems \ref{main} and \ref{main2}, and show that the minimal value for the diffusion we have identified in the statement is not optimal: the solution can apparently be global in time for 
$D<2C_2M_0$;
\item Illustrate the convergence for the particles approximation, and show that the actual rate of convergence as a function of $N$ seems to be much better than suggested by Theorem \ref{ratep}.
\end{enumerate}
For this purpose, we use a finite volume method introduced in \cite{CCH} to study drift-diffusion equations with gradient structure. Of course, there is no gradient structure in the present case, but the method can be adapted and it is proved to be robust. 
Let us briefly explain the principles of the approach.
We work on a Cartesian grid, with space steps $\Delta x,\Delta y>0$.
Given the time step $\Delta t>0$, we wish to update the numerical unknown with a finite volume formula which looks like
\[
\rho^{n+1}_{i,j}=\rho^{n+1}_{i,j}-\ds\frac{\Delta t}{\Delta x}(F_{i+1/2,j}-F_{i-1/2,j})-\ds\frac{\Delta t}{\Delta y}(G_{i,j+1/2}-G_{i,j-1/2})
\]
where we need to find a relevant definition for the numerical fluxes $F,G$.
To this end, we rewrite the right hand side of \eqref{pde} as 
\[\nabla\cdot\Big( \rho(\nabla\ln(\rho) -\vec F[\rho])\Big)
=\partial_x \Big(\rho(\partial_x \ln(\rho)+\partial_x U)\Big)+ \partial_y \Big(\rho(\partial_y \ln(\rho)+\partial_y V)\Big)
\]
where $U,V$ are the scalar functions defined by  
\[
U(x,y,t)=\ds\int |x-x'|\rho(x',y,t)\ud x',\qquad
V(x,y,t)=\ds\int |y-y'|\rho(x,y',t)\ud x'
.\]
We shall therefore apply the ideas in \cite{CCH} directionwise.
The flux $F_{i+1/2,j}$ is given by applying the upwinding principle with the ``velocity'' $\xi=\partial_x \ln(\rho)+\partial_x U$
which leads to 
\[
F_{i+1/2,j}=\big[\xi_{i+1/2,j}\big]_+\rho_{i,j} + \big[\xi_{i+1/2,j}\big]_+\rho_{i+1,j}.\]
The interface value is  obtained
by the mere centered difference
\[
\xi_{i+1/2,j}=\ds\frac{1}{\Delta x}\Big(\ln(\rho_{i+1,j})- \ln(\rho_{i,j}) + U_{i+1,j}-U_{i,j} \Big),
\]
where the integral that defines $U$ 
can be evaluated 
by a quadrature rule (the rectangle rule, say).
A similar construction applies to construct the flux $G$.
The accuracy of the method can be  improved by using 
a polynomial reconstruction of the density, with a suitable slope limiter, instead of the mere upwind scheme, in the spirit of the design of MUSCL schemes. We refer the reader to \cite{CCH} for further details and the analysis of this scheme for gradient--flow equations.
We can equally use a second-order Runge-Kutta method for the time integration. We do not explicitly introduce a regularization for the singular forces \eqref{def_f} in the code; we simply compute \eqref{def_f} by summing over rows or columns of the square grid. This corresponds to an effective regularization of the order of the grid spacing (typically $\Delta x=\Delta y=0.05$ in the simulations presented below). For the particles simulations, we integrate directly the regularized equations \eqref{eq:reg_part_x}--\eqref{eq:reg_part_y} by using the Euler method. We typically use $\epsilon=0.1$.

\begin{figure}
\centerline{\includegraphics[width=18cm]{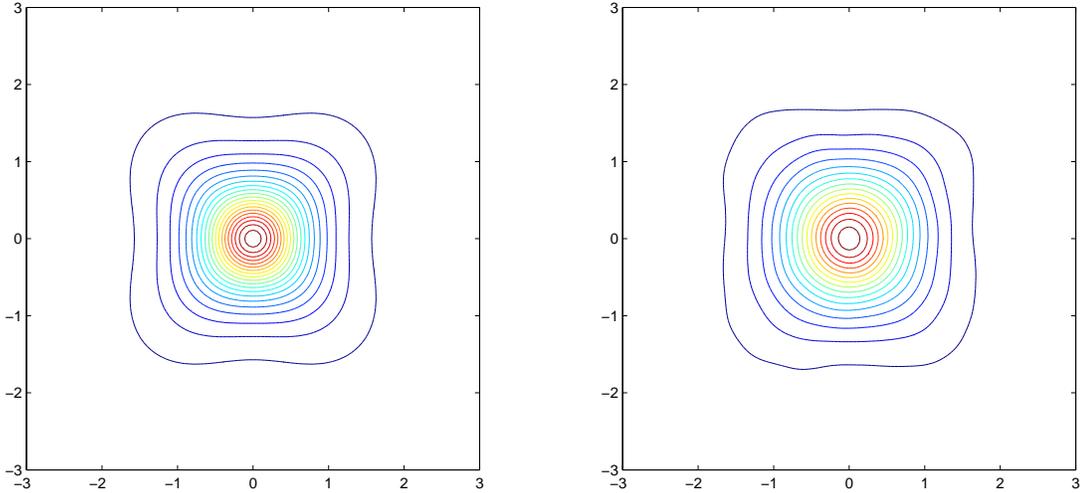}}
\caption{\label{fig:contour} Contour plot of the density $\rho$ for $D=0.15$, at time $t=5$. The left plot is done using the finite volume method introduced in \cite{CCH}. The right plot is done using the (mollified) particles approximation with 10 samples of $10^4$ particles. Note that the noise due to the finite number of particles is still visible.}
\end{figure}

\begin{figure}[htbp]
\centerline{\includegraphics[width=8cm]{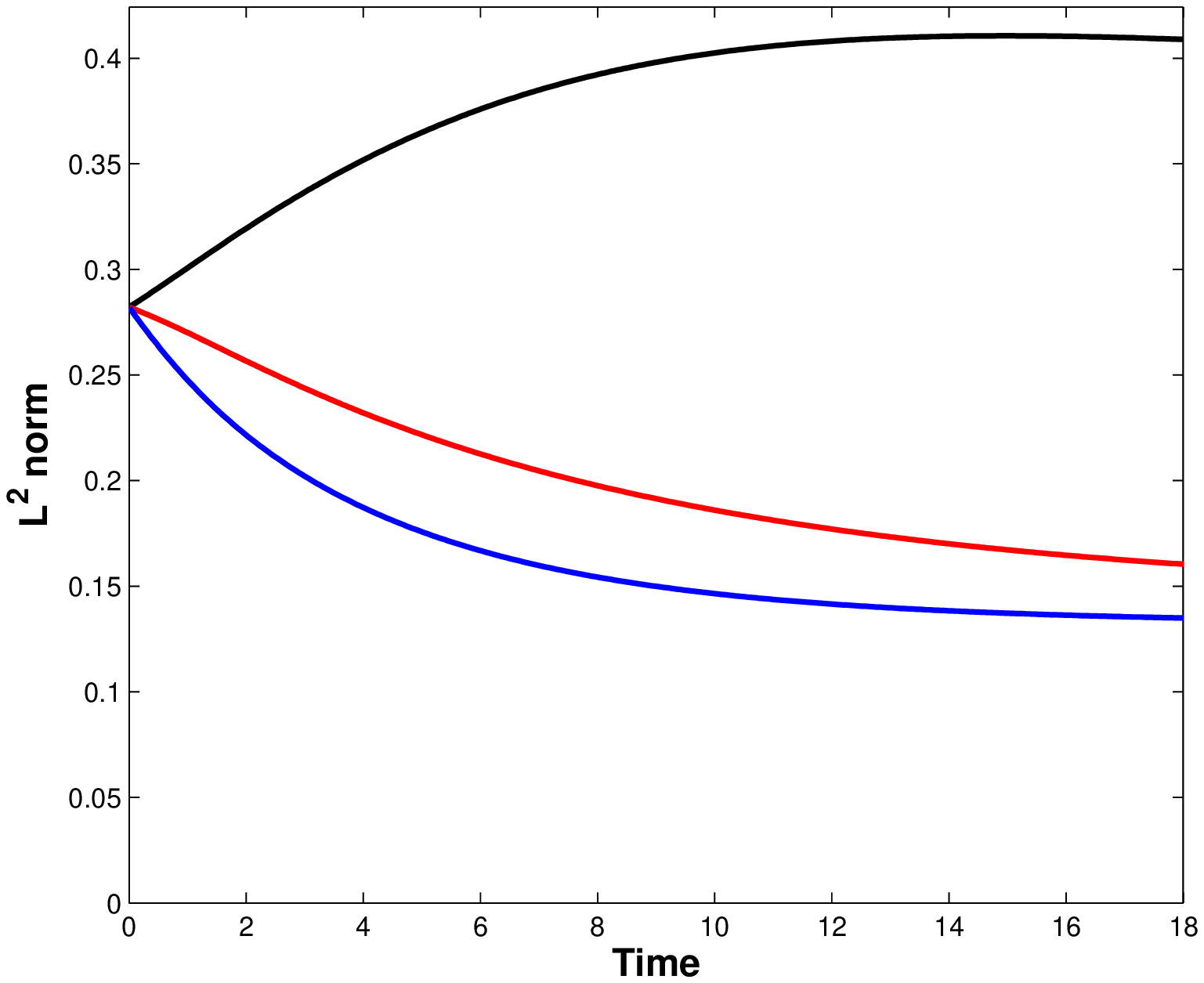}~\includegraphics[width=8cm]{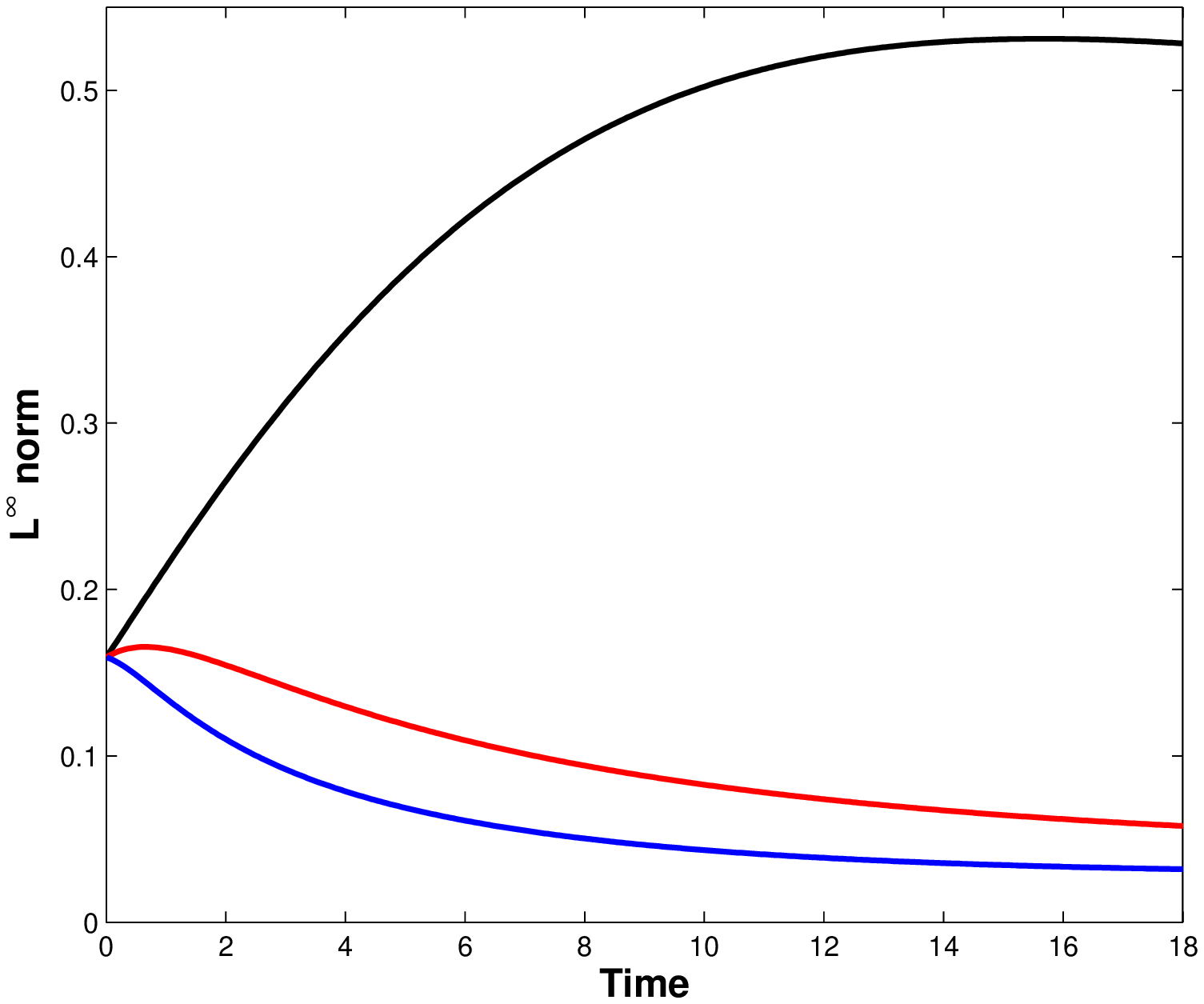}}
\caption{\label{fig:l2} $L^2$ norm (left) and 
$L^\infty$ norm (right)
as a function of time for $D=0.15$ (black), $D=0.25$ (red) and $D=0.35$ (blue).}
\end{figure}
Fig.~\ref{fig:contour} shows a contour plot of $\rho$ at late times for $D=0.15$ obtained by using the finite volume method introduced in \cite{CCH} (left plot) and the (mollified) particles approximation (right plot). Fig. \ref{fig:l2} 
shows the evolution of the $L^2$ and $L^\infty$ norms for various values of $D$. $D=0.15$ is smaller than $2C_2$, the threshold of Theorem \ref{main} (here $M_0=1$): the $L^2$ norm is not monotonically decreasing, but 
there  is apparently no finite time singularity.
\begin{figure}[htbp]
\centerline{\includegraphics[width=10cm]{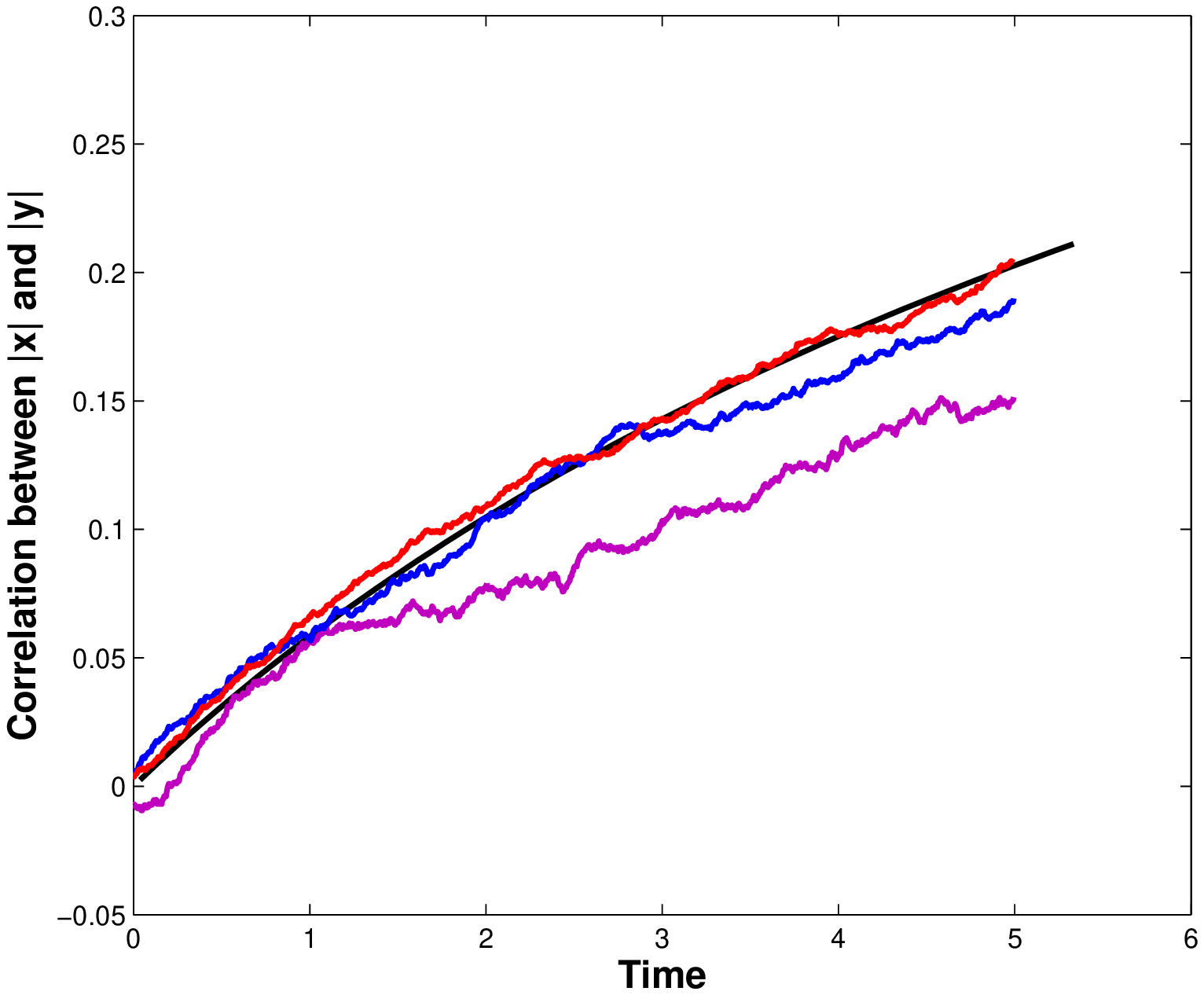}}
\caption{\label{fig:convergenceN} Plots of the quantity $<|xy|>-<|x|><|y|>$, where $<\cdot>$ stands for the integral with weight $\rho$. Comparison between the PDE solution (black line) and particles simulations with $N=2000$ (purple), $N=4000$ (blue) and $N=8000$ (red). There is always a single run for the particles simulations. The parameters are $D=0.15$, and for the particles simulation $\eps=0.1$.}
\end{figure}
Fig.~\ref{fig:convergenceN} shows that particles simulations are reasonably close to the PDE simulations already for a number of particles much smaller than that suggested by Theorem \ref{ratep}.


\section*{Acknowledgements}

We are gratefully indebted to Nicolas Fournier and Jos\'e Antonio Carrillo for many motivating discussions and helpful advices.
J. Barr\'e acknowledges the support of CNRS and Imperial College London which has made a visit of several months at the Math Department of ICL possible.

\bibliographystyle{plain}
\bibliography{BCGSeptember2016}

\end{document}